\documentclass[a4paper,11pt]{article}

\usepackage{a4wide}
\usepackage{graphicx}
\usepackage{amssymb,amsmath,amsthm,mathrsfs,xfrac,mathtools}
\usepackage{enumerate, enumitem}
\usepackage[title]{appendix}
\usepackage{amsfonts}
\usepackage{hyperref}
\usepackage[utf8]{inputenc}
\usepackage{color}
\usepackage{pgf,tikz}
\usetikzlibrary{arrows}

\newtheorem{proposition}{Proposition}[section]
\newtheorem{theorem}[proposition]{Theorem}
\newtheorem{lemma}[proposition]{Lemma}
\newtheorem{definition}[proposition]{Definition}
\newtheorem{corollary}[proposition]{Corollary}
\newtheorem{remark}[proposition]{Remark}

\numberwithin{equation}{section}
\allowdisplaybreaks[4]

\renewenvironment{proof}{\smallskip\noindent\emph{\textbf{Proof.}}%
  \hspace{1pt}}{\hspace{-5pt}{\nobreak\quad\nobreak\hfill\nobreak%
    $\square$\vspace{2pt}\par}\smallskip\goodbreak}

\newenvironment{proofof}[1]{\smallskip\noindent{\textbf{Proof~of~#1.}}%
  \hspace{1pt}}{\hspace{-5pt}{\nobreak\quad\nobreak\hfill\nobreak%
    $\square$\vspace{2pt}\par}\smallskip\goodbreak}

\newcommand{\C}[1]{\mathbf{C^{#1}}}
\newcommand{\Cc}[1]{\mathbf{C_c^{#1}}}
\renewcommand{\L}[1]{\mathbf{L^#1}}

\newcommand{\BV}{\mathbf{BV}}
\newcommand{\W}[2]{{\mathbf{W}^{#1,#2}}}
\newcommand{\modulo}[1]{{\left|#1\right|}}
\newcommand{\norma}[1]{{\left\|#1\right\|}}

\newcommand{\reali}{{\mathbb{R}}}
\newcommand{\R}{\mathbb R}
\newcommand{\naturali}{{\mathbb{N}}}
\newcommand{\N}{{\mathbb N}}
\newcommand{\interi}{{\mathbb{Z}}}

\renewcommand{\epsilon}{\varepsilon}
\renewcommand{\phi}{\varphi}

\renewcommand{\theta}{\vartheta}

\newcommand{\tv}{\mathinner{\rm TV}}

\newcommand{\sgnp}{\operatorname{sgn}^{+}}
\newcommand{\sgnm}{\operatorname{sgn}^{-}}

\renewcommand{\d}[1]{\mathinner{\mathrm{d}{#1}}}

\newcommand{\dt}{{\Delta t}}
\newcommand{\dx}{{\Delta x}}

\newcommand{\rh}[1]{\rho^{n}_{#1}}
\newcommand{\Rh}[1]{R^{n}_{#1}}

\DeclareMathOperator{\sgn}{sgn}

\newcommand{\del}{\partial}

\newcommand{\be}{\begin{equation}}
\newcommand{\ee}{\end{equation}}

\definecolor{ffqqqq}{rgb}{1.,0.,0.}
\definecolor{uuuuuu}{rgb}{0.26666666666666666,0.26666666666666666,0.26666666666666666}

\makeatletter
\let\@fnsymbol\@arabic
\makeatother

\makeatletter
\newcommand\appendix@section[1]{%
\refstepcounter{section}%
\orig@section*{Appendix \@Alph\c@section: #1}%
\addcontentsline{toc}{section}{Appendix \@Alph\c@section: #1}%
}
\let\orig@section\section
\g@addto@macro\appendix{\let\section\appendix@section}
\makeatother

\title{Well-posedness of IBVP for 1D\\scalar non-local conservation laws}

\author{Paola Goatin\footnotemark[1] \and Elena Rossi\footnotemark[1]}

\date{ }

\begin{document}
\maketitle
\footnotetext[1]{Inria Sophia Antipolis - M\'editerran\'ee,
  Universit\'e C\^ote d'Azur, Inria, CNRS, LJAD, 2004 route des
  Lucioles - BP 93, 06902 Sophia Antipolis Cedex, France. E-mail:
  \texttt{\{paola.goatin, elena.rossi\}@inria.fr}}

\begin{abstract}

  \noindent
  We consider the initial boundary value problem (IBVP) for a
  non-local scalar conservation laws in one space dimension. The
  non-local operator in the flux function is not a mere convolution
  product, but it is assumed to be \emph{aware of
    boundaries}. Introducing an adapted Lax-Friedrichs algorithm, we
  provide various estimates on the approximate solutions that allow to
  prove the existence of solutions to the original IBVP. The
  uniqueness follows from the Lipschitz continuous dependence on
  initial and boundary data, which is proved exploiting results
  available for the \emph{local} IBVP.

  \medskip

  \noindent\textit{2010~Mathematics Subject Classification: 35L04,
    35L65, 65M12, 65N08, 90B20}

  \medskip

\end{abstract}

\section{Introduction}

We consider the following Initial Boundary Value Problem (IBVP) on the
open bounded interval $]a,b[\, \subset \reali$
\begin{equation}\label{eq:original}
  \left\{
    \begin{array}{l@{\qquad}r@{\,}c@{\,}l}
      \del_t \rho + \d{}_x f(t,x,\rho, \mathcal{J}\rho) =0,
      & (t,x)
      &\in
      &\R^+\times ]a,b[,
      \\
      \rho(0,x) = \rho_o(x),
      & x
      &\in
      &]a,b[,
      \\
      \rho(t,a) = \rho_a (t),
      & t
      &\in
      &\R^+,
      \\
      \rho(t,b) = \rho_b (t),
      & t
      &\in
      &\R^+,
    \end{array}
  \right.
\end{equation}
where $\mathcal{J}$ denotes a non-local operator and we use the
notation
\begin{align}
  \nonumber
  \d{}_x f\left(t,x,\rho (t,x), \left(\mathcal{J}\rho (t)\right) \!(x)\right) = \
  & \partial_x f\left(t,x,\rho (t,x), \left(\mathcal{J}\rho (t)\right) \!(x)\right)
  \\ \label{eq:Dx}
  & + \partial_\rho f\left(t,x,\rho (t,x), \left(\mathcal{J}\rho (t)\right) \!(x)\right) \, \partial_x \rho (t,x)
  \\ \nonumber
  & + \partial_R f\left(t,x,\rho (t,x),\left(\mathcal{J}\rho (t)\right)\! (x)\right) \,
    \partial_x \left(\mathcal{J}\rho (t)\right) \!(x).
\end{align}
The same problem was studied in~\cite{DFG}. In that case, the choice
for $\mathcal{J}$ is the classical convolution product
$\mathcal{J}\rho = \rho * \eta$, $\eta$ being a smooth convolution
kernel. However, in such formulation, the non-local term may exceed
the boundaries of the spatial domain. The authors address this issue
by extending the solution outside the spatial domain, setting it
constantly equal to the corresponding boundary condition value.

Here, we propose a different approach. We follow the treatment of the
boundary conditions proposed in~\cite{multipop}, where a particular
multi-dimensional system of conservation laws in bounded domains with
zero boundary conditions is considered.  More precisely, a non-local
operator aware of the presence of boundaries is introduced. In the
present one-dimensional setting, this reads
\begin{equation}
  \label{eq:nl}
  \left(\mathcal{J}\rho (t)\right) \!(x) = \frac{1}{W (x)} \int_a^b \rho (t,y) \, \omega (y-x)\d{y},
  \quad   \mbox{with} \quad
  W (x) = \int_a^b \omega (y-x) \d{y},
\end{equation}
for a suitable convolution kernel $\omega$.

\smallskip

In recent years, the literature on non-local conservation laws has
widely increased. These equations are indeed used to model various
physical phenomena: from sedimentation models~\cite{sedimentation} to
granular flow~\cite{granular}, from vehicular
traffic~\cite{BlandinGoatin} to crowd
dynamics~\cite{Carrillo,ColomboGaravelloMercier2012,ColomboLecureuxPerDafermos},
from conveyor belts~\cite{matflow} to supply chains~\cite{ArmbrusterDegondRinghofer}.\\
Although physically those models might be defined in a bounded domain
and numerical integrations require it as well, they have been mostly
studied in the whole space $\reali$ or $\reali^n$. The main difficulty
lies indeed in the fact that the non-local operator may need to
evaluate the unknown outside the boundaries of the spatial domain,
where it is not defined.
\\
The analysis of the non-local problem~\eqref{eq:original} is carried
out exploiting the same strategy used in both~\cite{DFG}
and~\cite{1d}.  As already mentioned, \cite{DFG} studies the non local
IBVP~\eqref{eq:original} where the non-local operator is the standard
convolution product, while~\cite{1d} considers the \emph{local}
problem for a balance law, i.e.~a one dimensional IBVP where the flux
function has the form $f(t,x,\rho)$ and there is also a source term.
We remark that it could be possible to use the results of~\cite{1d} to
study the non-local problem~\eqref{eq:original}: indeed, the link
between the two problems is obtained by defining the \emph{local} flux
by $\tilde f(t,x,\rho) = f(t,x,\rho, \mathcal{J}\rho)$, where
$\mathcal{J}\rho = \left(\mathcal{J}\rho (t)\right) \!(x)$.  However,
in this way the \emph{a priori} estimates on the solution would be
less precise than those presented in this work.  Namely, a positivity
result and an $\L1$-bound on the solution are missing in~\cite{1d}.
Moreover, $\L\infty$-estimate recovered here depends on the first
derivatives of the flux function, see Theorem~\ref{thm:main}, while
using the results of~\cite{1d} yields an estimate depending on the
mixed second derivatives of $f$.
\\
Nevertheless, the result concerning the stability with respect to the
flux function proved in~\cite{1d}, recalled below in
Theorem~\ref{thm:bohStab}, is of crucial importance in this work,
since it contributes significantly in the proof of the Lipschitz
continuous dependence of solutions to~\eqref{eq:original} on initial
and boundary data, see Proposition~\ref{prop:lipDep}, and thus in the
proof of the uniqueness of solution to~\eqref{eq:original}.  At this
regard, we remark that the stability proof provided in~\cite{DFG} is
wrong, but could be fixed following the same strategy proposed here.

\smallskip

The paper is organised as follows. Section~\ref{sec:MR} presents the
assumptions needed on problem~\eqref{eq:original} and the main result
of this paper, whose proof is postponed to Section~\ref{sec:proof}.
Section~\ref{sec:existence} is devoted to the introduction of the
finite volume approximation of problem~\eqref{eq:original} and its
analysis.  The Lipschitz continuous dependence of solutions
to~\eqref{eq:original} on initial and boundary data is proved in
Section~\ref{sec:lipDep}. The final appendix~\ref{sec:app} recalls
some results from~\cite{1d} on the \emph{local} IBVP, necessary
throughout the paper.

\section{Main results}
\label{sec:MR}

We introduce the following notation:
\begin{displaymath}
  \label{eq:sgn}
  \sgnp (s)
  =
  \begin{cases}
    1 & \mbox{if } s>0,
    \\
    0 & \mbox{if } s \leq 0,
  \end{cases}
  \qquad
  \sgnm (s)
  =
  \begin{cases}
    0 & \mbox{if } s \geq 0,
    \\
    -1 & \mbox{if } s < 0,
  \end{cases}
  \qquad
  \begin{array}{r@{\;}c@{\;}l}
    s^+
    & =
    & \max\{s, 0\} ,
    \\
    s^-
    & =
    & \max\{-s, 0\} .
  \end{array}
\end{displaymath}
In the rest of the paper, we will denote
$\mathcal{I} (r,s) = [\min\left\{r,s\right\}, \max\left\{r,s
\right\}]$, for any $r, s \in \R$.

\noindent We make the following assumptions on the flux function $f$
and on the convolution kernel $\omega$:
\begin{enumerate}[label={$\boldsymbol{(f)}$}]
\item
  \label{f}$f\in \C2 (\reali^+ \times [a,b] \times \reali \times
  \reali; \reali)$ and there exist $L, C>0$ such that
  \begin{align*}
    &f(t,x,0,R) =  0
    &
    &\mbox{for }  t \in \reali^+,\,  x \in [a,b], \, R \in \reali,
    \\
    & \sup_{t,x,\rho,R}\modulo{\partial_\rho f (t,x,\rho,R)} < L,
    &&
    \\
    & \sup_{t,x,R}\modulo{\partial_x f (t,x,\rho,R)} < C \modulo{\rho},
    &
    & \sup_{t,x,R}\modulo{\partial_R f (t,x,\rho,R)} < C \modulo{\rho},
    \\
    & \sup_{t,x,R}\modulo{\partial_{xx}^2 f (t,x,\rho,R)} < C \modulo{\rho},
    &
    & \sup_{t,x,R}\modulo{\partial_{xR}^2 f (t,x,\rho,R)} < C \modulo{\rho},
    \\
    & \sup_{t,x,R}\modulo{\partial_{RR}^2 f (t,x,\rho,R)} < C \modulo{\rho}.
    &&
  \end{align*}
\end{enumerate}
\begin{enumerate}[label={$\boldsymbol{(\omega)}$}]
\item \label{omega}
  $\omega \in (\C2 \cap \W21 \cap \W2\infty) (\reali;\reali)$ is such
  that
  \begin{displaymath}
    \int_\reali \omega (y) \d{y} = 1
  \end{displaymath}
  and there exists $K_\omega > 0$ such that for all $x \in [a,b]$
  \begin{equation}
    \label{eq:W}
    W (x) =  \int_a^b \omega (y-x) \d{y} \geq K_\omega.
  \end{equation}
\end{enumerate}
The requirement~\eqref{eq:W} guarantees that $\mathcal{J}$
in~\eqref{eq:nl} is well defined for all $x\in\, ]a,b[$.

\smallskip

We recall below two different definitions of solution to
problem~\eqref{eq:original}.  Recall that the two definitions are
equivalent for functions in
$(\L\infty \cap \BV) (\R^+\times]a,b[; \R)$. We refer
to~\cite{definizioni} for further details on the link between this two
definitions.

The first definition follows from~\cite{BardosLerouxNedelec}.
\begin{definition}
  \label{def:solBLN} A function
  $\rho\in (\L\infty \cap \BV) (\R^+\times]a,b[; \R)$ is an entropy
  weak solution to problem~\eqref{eq:original} if, for all
  $\phi\in\Cc1(\R^2;\R^+)$ and $k\in\R$,
  \begin{align}
    \nonumber
    & \int_0^{+\infty} \int_{a}^{b}  \left\{ \modulo{\rho  -k} \partial_t \phi (t,x)
      + \sgn\left(\rho  - k\right) \left[
      f\left(t, x, \rho, R (t,x)\right) - f\left(t, x, k, R (t,x)\right)
      \right] \, \partial_x \phi (t,x)
      \right.
    \\
    \nonumber
    & \qquad\qquad \left.
      - \sgn\left(\rho -k\right) \left( \partial_x f \left(t, x, k, R (t,x)\right)
      + \partial_R f\left(t, x, k ,R (t,x)\right) \, \partial_x R (t,x)
      \right)\,  \phi (t,x) \right\}
      \d{x} \d{t}
    \\
    \label{eq:BLN}
    & + \int_{a}^{b}  \modulo{\rho_o(x)-k}  \phi(0,x)  \d{x}
    \\ \nonumber
    & + \int_0^{+\infty}
      \sgn \left(  \rho_a(t)  - k \right)
      \left[
      f \left(t, a, \rho (t, a^+), R(t,a)\right) - f\left(t, a, k, R (t,a)\right)
      \right]\,  \phi(t,a)  \d{t}
    \\
    \nonumber
    & - \int_0^{+\infty}
      \sgn \left(  \rho_b(t)  - k \right)
      \left[
      f\left(t, b, \rho (t,b^-), R (t,b)\right) - f \left(t, b, k, R(t,b)\right)
      \right] \,\phi(t,b) \d{t}
      \geq  0,
  \end{align}
  where, for $x \in [a,b]$,
  \begin{equation}
    \label{eq:R}
    R (t,x) = \left(\mathcal{J}\rho (t)\right)\! (x) =
    \frac{1}{W (x)} \int_a^b \rho (t,y) \, \omega (y-x) \d{y},
  \end{equation}
  and $W$ is as in~\eqref{eq:nl}.
\end{definition}
The second definition was introduced in~\cite{Martin,Vovelle2002}.
\begin{definition}
  \label{def:MV}
  A function $\rho \in \L\infty(\reali^+ \times ]a,b[; \R)$ is an
  entropy weak solution to problem~\eqref{eq:original} if, for all
  $\phi \in \Cc1 (\R^2; \R^+)$ and $k \in \R$,
  \begin{align}
    \nonumber
    & \int_0^{+\infty} \!\!\!\int_{a}^{b} \! \left\{ \left(\rho  -k\right)^\pm \partial_t \phi (t,x)
      + \sgn^\pm\left(\rho  - k\right) \left[
      f\left(t, x, \rho, R (t,x)\right) - f\left(t, x, k, R (t,x)\right)
      \right] \, \partial_x \phi (t,x)
      \right.
    \\
    \nonumber
    & \qquad\qquad \left.
      - \sgn^\pm\!\left(\rho -k\right)  \left( \partial_x f \left(t, x, k, R (t,x)\right)\!
      + \partial_R f\left(t, x, k ,R (t,x)\right) \, \partial_x R (t,x)
      \right)   \phi (t,x) \right\}
      \d{x} \d{t}
    \\
    \label{eq:MV}
    & + \int_{a}^{b}  \left(\rho_o(x)-k\right)^\pm  \phi(0,x)  \d{x}
    \\ \nonumber
    & + \norma{\partial_\rho f}_{\L\infty}
      \left( \int_0^{+\infty}
      \left(\rho_a(t)  - k \right)^\pm  \phi(t,a)  \d{t}
      + \int_0^{+\infty}
      \left( \rho_b(t)  - k \right)^\pm \phi(t,b) \d{t}\right)
      \geq  0,
  \end{align}
  where $R$ is as in~\eqref{eq:R} and
  $ \norma{\partial_\rho f}_{\L\infty} =\displaystyle \sup_{(t,x) \in
    \reali^+\times [a,b]} \modulo{\partial_\rho f\left(t, x, \rho
      (t,x), R (t,x)\right)}$.
\end{definition}

We can now state our main result.

\begin{theorem}
  \label{thm:main}
  Let~$\boldsymbol{(f)}$ and $\boldsymbol{(\omega)}$ hold. Let
  $\rho_o \in \BV (\,]a,b[; \reali^+)$ and
  $\rho_a, \, \rho_b \in \BV(\reali^+; \reali^+)$.  Then, for all
  $T>0$, problem~\eqref{eq:original} has a unique entropy weak
  solution
  $\rho \in (\L1 \cap \L\infty \cap \BV) ([0,T] \times
  ]a,b[;\reali^+)$. Moreover, the following estimates hold: for any
  $t \in [0,T]$,
  \begin{align*}
    \norma{\rho (t)}_{\L1 ([a,b])}
    \leq\
    &  \mathcal{R}_1(t),
    \\
    \norma{\rho (t)}_{\L\infty(]a,b[)}
    \leq \
    &  \mathcal{R}_\infty(t),
    \\
    \tv (\rho (t))
    \leq \
    &  e^{t \, \mathcal{T}_1 (t)} \left(\tv (\rho_o) + \tv (\rho_a; [0,t]) + \tv (\rho_b;[0,t])\right)
      + \frac{\mathcal{T}_2 (t)}{\mathcal{T}_1 (t)} (e^{t \, \mathcal{T}_1 (t)} -1),
  \end{align*}
  and, for $\tau > 0$,
  \begin{displaymath}
    \norma{\rho (t) - \rho (t-\tau)}_{\L1 ([a,b])}
    \leq
    \tau \left(\mathcal{C}_t (t) + 3 \, L \left(\tv(\rho_a;[t-\tau,t])+ \tv(\rho_b;[t-\tau,t])\right)
    \right),
  \end{displaymath}
  where
  \begin{align}
    \label{eq:R1}
    \mathcal{R}_1(t) = \
    & \norma{\rho_o}_{\L1 ([a,b])} + \norma{\partial_\rho f}_{\L\infty}
      \, \left(\norma{\rho_a}_{\L1 ([0,t])} + \norma{\rho_b}_{\L1 ([0,t])}\right),
    \\
    \label{eq:Rinf}
    \mathcal{R}_\infty(t) = \
    & e^{t \, C \, \left(1 + \mathcal{L} \, \mathcal{R}_1 (t)\right)}
      \max\left\{\norma{\rho_o}_{\L\infty (]a,b[)}, \, \norma{\rho_a}_{\L\infty ([0,t])}, \,
      \norma{\rho_b}_{\L\infty ([0,t])}\right\},
    \\
    \label{eq:T1}
    \mathcal{T}_1(t) =\
    & \norma{\partial_{\rho x}^2 f}_{\L\infty ([0,t] \times [a,b]\times \reali^2 )}
      +  \mathcal{L}\, \mathcal{R}_1 (t) \,
      \norma{\partial_{\rho R}^2 f}_{\L\infty ([0,t] \times [a,b]\times \reali^2 )},
    \\
    \mathcal{T}_2(t) = \
    & \mathcal{K}_2 (t)
      + \frac32 \, C (1 + \mathcal{L}\, \mathcal{R}_1 (t)) \,
      \mathcal{R}_\infty (t)
      + \left[\mathcal{K}_3 (t) +\frac{C}{2} \, (1 + \mathcal{L}\,
      \mathcal{R}_1 (t)) \right] \norma{\rho_a}_{\L\infty ([0,t])},
  \end{align}
  with $\mathcal{L}$ as in~\eqref{eq:L},
  $\mathcal{K}_2 (t), \mathcal{K}_3 (t)$ as in~\eqref{eq:K123}, with
  $\mathcal{C}_1 (t)$ substituted by $\mathcal{R}_1 (t)$, and
  $\mathcal{C}_t (t)$ is as in~\eqref{eq:Ct}, with $\alpha=L$.
\end{theorem}

\section{Existence of weak entropy solutions}
\label{sec:existence}

Fix $T>0$. Fix a space step $\dx$ such that $b-a = N \Delta x$, with
$N \in\N$, and a time step $\dt$ subject to a CFL condition, specified
later.  Introduce the following notation
\begin{align*}
  y_{k} := \ & (k-1/2) \dx,
  &
    y_{k+1/2} := \ & k \dx
  &
    \mbox{for } & k \in \interi,
  \\
  x_{j+1/2}:= \ & a + j \dx = a + y_{j + 1/2},
  & &&
       \mbox{for } & j = 0, \ldots, N,
  \\
  x_j:= \ & a + (j-1/2)\dx = a + y_j,
  & &&
       \mbox{for } & j = 1, \ldots, N,
\end{align*}
where $x_{j+1/2}$, $j=0, \ldots, N$, are the cells interfaces and
$x_j$, $j=1, \ldots, N$, the cells centres. Moreover, set
$N_T =\lfloor T/\dt \rfloor$ and, for $n =0, \ldots, N_T$ let
$t^n=n\dt$ be the time mesh. Set $\lambda=\dt / \dx$.

Approximate the initial datum $\rho_o$ and the boundary data as
follows:
\begin{align*}
  &\rho_j^0 := \frac{1}{\dx} \int_{x_{j-1/2}}^{x_{j+1/2}} \rho_o(x) \d{x},
    \qquad j=1,\ldots,N,
  \\
  &\rho_a^n
    := \frac{1}{\dt} \int_{t^n}^{t^{n+1}} \rho_a(t) \d{t},
    \quad
    \rho_b^n := \frac{1}{\dt} \int_{t^n}^{t^{n+1}} \rho_b(t) \d{t},
    \qquad n=0, \ldots, N_T -1.
\end{align*}
Introduce moreover the notation $\rho_0^n = \rho_a^n$ and
$\rho_{N+1}^n = \rho_b^n$. For $n =0, \ldots, N_T-1$, set
\begin{align}
  \nonumber
  & \omega^k := \omega (y_k) = \omega \left((k-1/2) \dx\right)
  &  \mbox{ for }
  & k \in \interi,
  \\
  \label{eq:5}
  &W_{j+1/2} := \dx \, \sum_{k=1}^N \omega^{k-j}
  &  \mbox{ for }
  & j=0, \ldots, N,
  \\
  \nonumber
  & \Rh{j+1/2} := \frac{\dx}{W_{j+1/2}} \,
    \displaystyle \sum_{k=1}^N \omega^{k-j} \, \rh{k}
  & \mbox{ for }
  & j= \  0,\ldots,N.
\end{align}

Introduce the following modified Lax-Friedrichs flux adapted to the
present setting: for $n =0, \ldots, N_T-1$ and $j=0, \ldots, N$,
\begin{equation}
  \label{eq:numFl}
  F^n_{j+1/2} (\rh{j}, \rh{j+1}) =
  \frac{1}{2}\!
  \left[
    f (t^n, x_{j+1/2}, \rh{j}, \Rh{j+1/2}) +  f (t^n, x_{j+1/2}, \rh{j+1}, \Rh{j+1/2})
    -  \alpha \!\left(\rh{j+1} - \rh{j} \right)\! \right] ,
\end{equation}
where $\alpha \geq 1$ is the viscosity coefficient.

We define a piecewise constant approximate solution $\rho_{\Delta}$
to~\eqref{eq:original} as
\begin{equation}
  \label{eq:rhodelta}
  \rho_{\Delta} (t,x) =  \rho^{n}_{j}
  \quad \mbox{ for } \quad
  \left\{
    \begin{array}{r@{\;}c@{\;}l}
      t & \in &[t^n, t^{n+1}[ \,,
      \\
      x & \in & [x_{j-1/2}, x_{j+1/2}[ \,,
    \end{array}
  \right.
  \quad \mbox{ where } \quad
  \begin{array}{r@{\;}c@{\;}l}
    n & = & 0, \ldots, N_T-1,
    \\
    j & = &  1,\ldots,N,
  \end{array}
\end{equation}
through the finite volume scheme
\begin{equation}
  \label{eq:scheme}
  \rho^{n+1}_{j} =
  \rho_j ^n - \lambda \left(F_{j+1/2}^n (\rh{j}, \rh{j+1}) - F_{j-1/2}^n (\rh{j-1}, \rh{j})\right).
\end{equation}

\begin{remark}
  Concerning the first formula in~\eqref{eq:5}, observe that a
  different (more accurate) choice for the approximation of the kernel
  function $\omega$ is possible: indeed, one may define
  \begin{displaymath}
    \omega^k =
    \frac{1}{\Delta x}\int_{(k-1)\Delta x}^{k\Delta x}\omega (y) \d{y},
  \end{displaymath}
  which ensures that
  $W_{j+1/2}=\dx \sum_{k=1}^N\omega^{k-j}=W(x_{j+1/2})$.  This choice
  wouldn't result in any relevant change in the estimates derived in
  this paper.
\end{remark}

\subsection{Positivity}
\label{sec:pos}

In the case of positive initial and boundary data, we prove that under
a suitable CFL condition the scheme~\eqref{eq:scheme} preserves the
positivity.

\begin{lemma}
  \label{lem:pos}
  Let $\rho_o \in \L\infty (\,]a,b[; \reali^+)$ and
  $\rho_a, \, \rho_b \in \L\infty (\reali^+; \reali^+)$. Let~\ref{f}
  and~\ref{omega} hold. Assume that
  \begin{align}
    \label{eq:CFL}
    \alpha \geq \
    & L,
    &
      \lambda \leq\frac13 \, \min\left\{\frac1\alpha, \frac{1}{2 \, L + C \, \dx}\right\}.
  \end{align}
  Then, for all $t>0$ and $x \in \, ]a,b[$, the piecewise constant
  approximate solution $\rho_\Delta$~\eqref{eq:rhodelta} is such that
  $\rho_\Delta (t,x)\geq 0$.
\end{lemma}
\begin{proof}
  We closely follow~\cite[Lemma~1]{DFG}.  Fix $j$ between $1$ and $N$,
  $n$ between $0$ and $N_T -1$. Suppose that $\rh{j}\geq 0$ for all
  $j=1, \ldots,N$.  Rewrite~\eqref{eq:scheme} as follows:
  \begin{align*}
    \rho_j^{n+1} = \
    & \rh{j} - \lambda \left[
      F_{j+1/2}^n (\rh{j}, \rh{j+1}) \pm F_{j+1/2}^n (\rh{j}, \rh{j})
      \pm F_{j-1/2}^n (\rh{j}, \rh{j})- F_{j-1/2}^n (\rh{j-1}, \rh{j})
      \right]
    \\
    = \
    &
      (1 - \beta_j^n - \gamma_j^n) \, \rh{j} + \beta_j^n \, \rh{j-1} + \gamma_j^n \, \rh{j+1}
      - \lambda \left(F_{j+1/2}^n (\rh{j}, \rh{j})  - F_{j-1/2}^n (\rh{j}, \rh{j}) \right),
  \end{align*}
  with
  \begin{align}
    \label{eq:betajn}
    \beta_j^n = \
    &
      \begin{cases}
        \lambda \, \dfrac{F_{j-1/2}^n (\rh{j}, \rh{j})- F_{j-1/2}^n
          (\rh{j-1}, \rh{j})}{\rh{j}-\rh{j-1}} & \mbox{if } \rh{j}
        \neq \rh{j-1},
        \\
        0 & \mbox{if } \rh{j} = \rh{j-1},
      \end{cases}
    \\
    \label{eq:gammajn}
    \gamma_j^n = \
    &
      \begin{cases}
        - \lambda \, \dfrac{F_{j+1/2}^n (\rh{j}, \rh{j+1})-
          F_{j+1/2}^n (\rh{j}, \rh{j})}{\rh{j+1}-\rh{j}} & \mbox{if }
        \rh{j} \neq \rh{j+1},
        \\
        0 & \mbox{if } \rh{j} = \rh{j+1}.
      \end{cases}
  \end{align}
  Using the explicit expression of the numerical flux~\eqref{eq:numFl}
  and~\ref{f}, we obtain
  \begin{align*}
    & \modulo{F_{j+1/2}^n (\rh{j}, \rh{j})  - F_{j-1/2}^n (\rh{j}, \rh{j})}
    \\
    = \
    & \modulo{ f (t^n, x_{j+1/2}, \rh{j}, \Rh{j+1/2}) -  f (t^n, x_{j-1/2}, \rh{j}, \Rh{j-1/2})
      \pm f (t^n, x_{j-1/2}, \rh{j}, \Rh{j+1/2}) }
    \\
    \leq \
    &  \modulo{ f (t^n, x_{j+1/2}, \rh{j}, \Rh{j+1/2}) - f (t^n, x_{j-1/2}, \rh{j}, \Rh{j+1/2}) }
    \\
    &  + \modulo{f (t^n, x_{j-1/2}, \rh{j}, \Rh{j+1/2}) -  f (t^n, x_{j-1/2}, \rh{j}, \Rh{j-1/2})}
    \\
    \leq \
    & C \,\rh{j} \, \dx
      + \modulo{f (t^n, x_{j-1/2}, \rh{j}, \Rh{j+1/2}) - f (t^n, x_{j-1/2}, 0, \Rh{j+1/2}) }
    \\
    & + \modulo{ f (t^n, x_{j-1/2}, \rh{j}, \Rh{j-1/2}) -  f (t^n, x_{j-1/2}, 0, \Rh{j-1/2})}
    \\
    \leq \
    & C \, \rh{j} \, \dx + 2 \, L \, \rh{j}.
  \end{align*}
  Observe that, whenever $\rh{j} \neq \rh{j-1}$,
  \begin{align*}
    \beta_j^n =  \
    &\frac{\lambda}{2 \, (\rh{j}-\rh{j-1})}\left[
      f (t^n, x_{j-1/2}, \rh{j}, \Rh{j-1/2}) - f (t^n, x_{j-1/2}, \rh{j-1}, \Rh{j-1/2}) + \alpha \, (\rh{j}-\rh{j-1})
      \right]
    \\
    = \
    & \frac{\lambda}{2} \left(\partial_\rho f (t^n, x_{j-1/2}, r_{j-1/2}^n, \Rh{j-1/2}) + \alpha \right),
  \end{align*}
  with $r_{j-1/2}^n \in \mathcal{I}\left(\rh{j-1},
    \rh{j}\right)$. Similarly, whenever $\rh{j} \neq \rh{j+1}$,
  \begin{align*}
    \gamma_j^n =  \
    & -\frac{\lambda}{2 \, (\rh{j+1}-\rh{j})}\left[
      f (t^n, x_{j+1/2}, \rh{j+1}, \Rh{j+1/2}) - f (t^n, x_{j+1/2}, \rh{j}, \Rh{j+1/2}) - \alpha \, (\rh{j+1}-\rh{j})
      \right]
    \\
    = \
    & \frac{\lambda}{2} \left(\alpha - \partial_\rho f (t^n, x_{j+1/2}, r_{j+1/2}^n, \Rh{j+1/2}) \right),
  \end{align*}
  with $r_{j+1/2}^n \in \mathcal{I}\left(\rh{j}, \rh{j+1}\right)$.  By
  the conditions~\eqref{eq:CFL}, we get
  \begin{align*}
    \beta_j^n, \, \gamma_j^n \in
    & \left[0,\frac13\right],
    &
      (1-\beta_j^n - \gamma_j^n) \in
    & \left[\frac13, 1\right],
    &
      \lambda \, \modulo{F_{j+1/2}^n (\rh{j}, \rh{j})  - F_{j-1/2}^n (\rh{j}, \rh{j})} \leq
    & \frac13 \, \rh{j},
  \end{align*}
  which, using the inductive hypothesis, leads to
  \begin{displaymath}
    \rho_j^{n+1}
    \geq
    (1-\beta_j^n - \gamma_j^n) \rh{j}
    + \beta_j^n \, \rh{j-1} + \gamma_j^n \, \rh{j+1} - \frac13 \, \rh{j}
    \geq
    \frac13 \, \rh{j} - \frac13 \, \rh{j}
    \geq  0.
  \end{displaymath}
\end{proof}

\subsection{\texorpdfstring{$\L1$}{L 1} bound}
\label{sec:l1}

\begin{lemma}
  \label{lem:l1}
  Let $\rho_o \in \L\infty (\,]a,b[; \reali^+)$ and
  $\rho_a, \, \rho_b \in \L\infty (\reali^+; \reali^+)$. Let~\ref{f},
  \ref{omega} and~\eqref{eq:CFL} hold.  Then, for all $t>0$,
  $\rho_\Delta$ in~\eqref{eq:rhodelta} satisfies
  \begin{equation}
    \label{eq:l1}
    \norma{\rho_\Delta (t, \cdot)}_{\L1 (]a,b[)}
    \leq \mathcal{C}_1 (t),
  \end{equation}
  where
  \begin{equation}
    \label{eq:C1}
    \mathcal{C}_1 (t)=
    \norma{\rho_o}_{\L1 (]a,b[)} + \alpha \, \left(\norma{\rho_a}_{\L1 ([0,t])} + \norma{\rho_b}_{\L1 ([0,t])}\right).
  \end{equation}
\end{lemma}

\begin{proof}
  By Lemma~\ref{lem:l1}, we know that the scheme~\eqref{eq:scheme}
  preserves the positivity. Therefore, for $n=0, \ldots, N_T-1$,
  compute
  \begin{align*}
    \|\rho^{n+1}\|_{\L1 (]a,b[)} = \
    & \dx \, \sum_{j=1}^N \rho_j^{n+1}
    \\
    = \
    & \dx \sum_{j=1}^N  \left[
      \rh{j} - \lambda \left(F^n_{j+1/2} (\rh{j}, \rh{j+1}) - F^n_{j-1/2} (\rh{j-1}, \rh{j})\right)
      \right]
    \\
    = \
    & \dx \sum_{j=1}^N \rh{j} - \lambda \, \dx \, \left(
      F^n_{N+1/2} (\rh{N}, \rh{N+1}) - F^n_{1/2} (\rh{0}, \rh{1})
      \right)
    \\
    = \
    & \| \rho^n\|_{\L1 ([a.b])}
    \\
    & - \frac{\dt}{2}\!
      \left[
      f (t^n, x_{N+1/2}, \rh{N}, \Rh{N+1/2}) \!+ \!  f (t^n, x_{N+1/2}, \rh{b}, \Rh{N+1/2})
      - \alpha\!  \left(\rh{b} - \rh{N}\right)
      \right]
    \\
    & + \frac{\dt}{2}\,
      \left(
      f (t^n, x_{1/2}, \rh{a}, \Rh{1/2}) + f (t^n, x_{1/2}, \rh{1}, \Rh{1/2})
      - \alpha  \left(\rh{1} - \rh{a}\right)
      \right)
    \\
    = \
    &\| \rho^n\|_{\L1 ([a.b])}
      + \frac{\dt}{2}\, \left(
      - \alpha - \partial_\rho f (t^n, x_{N+ 1/2}, r^n_{N,0}, \Rh{N+1/2})
      \right)  \rh{N}
    \\
    & + \frac{\dt}{2}\, \left(
      \alpha - \partial_\rho f (t^n, x_{N+ 1/2}, r^n_{b,0}, \Rh{N+1/2})
      \right)  \rh{b}
    \\
    & + \frac{\dt}{2}\, \left(
      \alpha + \partial_\rho f (t^n, x_{1/2}, r^n_{a,0}, \Rh{1/2})
      \right)  \rh{a}
    \\
    & + \frac{\dt}{2}\, \left(
      - \alpha + \partial_\rho f (t^n, x_{1/2}, r^n_{1,0}, \Rh{1/2})
      \right)  \rh{1},
  \end{align*}
  where $r_{N,0}^n \in \mathcal{I}\left(0, \rh{N}\right)$,
  $r_{b,0}^n \in \mathcal{I}\left(0, \rh{b}\right)$,
  $r_{a,0}^n \in \mathcal{I}\left(0, \rh{a}\right)$ and
  $r_{1,0}^n \in \mathcal{I}\left(0, \rh{1}\right)$. By~\ref{f} and
  the assumption~\eqref{eq:CFL} on $\alpha$, the coefficients of
  $\rh{N}$ and $\rh{1}$ are negative. Thus
  \begin{displaymath}
    \|\rho^{n+1}\|_{\L1 (]a,b[)}
    \leq
    \| \rho^n\|_{\L1 (]a.b[)} + \alpha \, \dt  \left(\rh{a} + \rh{b}\right).
  \end{displaymath}
  An iterative argument yields the thesis.
\end{proof}

\subsection{\texorpdfstring{$\L\infty$}{L infinity} bound}
\label{sec:linf}

\begin{lemma}
  \label{lem:linf}
  Let $\rho_o \in \L\infty (\,]a,b[; \reali^+)$ and
  $\rho_a, \, \rho_b \in \L\infty (\reali^+; \reali^+)$. Let~\ref{f},
  \ref{omega} and~\eqref{eq:CFL} hold.  Then, for all $t>0$,
  $\rho_\Delta$ in~\eqref{eq:rhodelta} satisfies
  \begin{equation}
    \label{eq:linf}
    \norma{\rho_\Delta (t, \cdot)}_{\L\infty (]a,b[)}
    \leq  \max\left\{\norma{\rho_o}_{\L\infty (]a,b[)}, \, \norma{\rho_a}_{\L\infty ([0,t])}, \, \norma{\rho_b}_{\L\infty ([0,t])}\right\}\, e^{\mathcal{C}_2 (t)  t},
  \end{equation}
  where $\mathcal{C}_2 (t)$ is given by~\eqref{eq:C2}.
\end{lemma}

\begin{proof}
  Fix $n$ between $0$ and $N_T-1$. For $j=1, \ldots N$,
  rearrange~\eqref{eq:scheme} as in Lemma~\ref{lem:pos}, with the
  notation~\eqref{eq:betajn}--\eqref{eq:gammajn}:
  \begin{equation}
    \label{eq:1}
    \rho_j^{n+1} =
    (1 - \beta_j^n - \gamma_j^n) \, \rh{j} + \beta_j^n \, \rh{j-1} + \gamma_j^n \, \rh{j+1}
    - \lambda \left(F_{j+1/2}^n (\rh{j}, \rh{j})  - F_{j-1/2}^n (\rh{j}, \rh{j}) \right).
  \end{equation}
  Compute
  \begin{align*}
    &  \modulo{F_{j+1/2}^n (\rh{j}, \rh{j})  - F_{j-1/2}^n (\rh{j}, \rh{j}) }
    \\
    = \
    & \modulo{ f (t^n, x_{j+1/2}, \rh{j}, \Rh{j+1/2}) -  f (t^n, x_{j-1/2}, \rh{j}, \Rh{j-1/2})
      \pm f (t^n, x_{j-1/2}, \rh{j}, \Rh{j+1/2})}
    \\
    \leq \
    & \modulo{\partial_x f (t^n, \tilde{x}_j, \rh{j}, \Rh{j+1/2})}  |x_{j+1/2} - x_{j-1/2}|
      +
      \modulo{\partial_R f (t^n, x_{j-1/2}, \rh{j}, \tilde{R}^n_{j})} \modulo{\Rh{j+1/2} - \Rh{j-1/2}} .
  \end{align*}
  By~\eqref{eq:5}, we have
  \begin{align*}
    &\modulo{\Rh{j+1/2} - \Rh{j-1/2}}
    \\
    =\
    &
      \modulo{\frac{\dx}{W_{j+1/2}} \left( \sum_{k=1}^N  \omega^{k-j} \, \rh{k}\right)
      -
      \frac{\dx}{W_{j-1/2}} \left( \sum_{k=1}^N  \omega^{k-j+1} \, \rh{k}\right)
      }
    \\
    \leq \
    & \frac{\dx}{|W_{j+1/2}|} \sum_{k=1}^N |\omega^{k-j} - \omega^{k-j+1}| \, \rh{k}
      + \dx \, \left( \sum_{k=1}^N  |\omega^{k-j+1}| \, \rh{k}\right)
      \modulo{\frac{1}{W_{j+1/2}} - \frac{1}{W_{j-1/2}}}
    \\
    \leq \
    & \frac{\dx} {K_\omega}  \, \sum_{k=1}^N \rh{k} \, \modulo{\int_{y_{k-j}}^{y_{k-j+1}} \omega' (y) \d{y}}
      + \frac{\dx^2} {K_\omega^2} \, \norma{\omega}_{\L\infty (\R)} \left(\sum_{k=1}^N \rh{k}\right)
      \modulo{\sum_{k=1}^N (\omega^{k-j+1} - \omega^{k-j})}
    \\
    \leq \
    & \frac{\dx}{K_\omega} \, \norma{\omega'}_{\L\infty (\R)}\, \norma{\rho^n}_{\L1 (]a,b[)}
      + \frac{\dx}{K_\omega^2} \,  \norma{\omega}_{\L\infty (\R)} \, \norma{\rho^n}_{\L1 (]a,b[)}
      \, \sum_{k=1}^N  \modulo{\int_{y_{k-j}}^{y_{k-j+1}} \omega' (y) \d{y}}
    \\
    \leq \
    &
      \frac{\dx}{K_\omega} \, \norma{\omega'}_{\L\infty (\R)}\, \norma{\rho^n}_{\L1 (]a,b[)}
      + \frac{\dx}{K_\omega^2} \,  \norma{\omega}_{\L\infty (\R)} \, \norma{\omega'}_{\L1 (\R)}
      \, \norma{\rho^n}_{\L1 (]a,b[)}
    \\
    \leq \
    & \dx \, \left(
      \frac{ \norma{\omega'}_{\L\infty (\R)}}{K_\omega} +
      \frac{ \norma{\omega}_{\L\infty (\R)} \norma{\omega'}_{\L1 (\R)}}{K_\omega^2}
      \right)
      \mathcal{C}_1 (t^n),
  \end{align*}
  where $\mathcal{C}_1 (t)$ is defined in~\eqref{eq:C1}. Setting
  \begin{equation}
    \label{eq:L}
    \mathcal{L} =
    \left(  \frac{ \norma{\omega'}_{\L\infty (\R)}}{K_\omega}
      + \frac{ \norma{\omega}_{\L\infty (\R)} \norma{\omega'}_{\L1 (\R)}}{K_\omega^2}
    \right),
  \end{equation}
  we obtain
  \begin{displaymath}
    \modulo{F_{j+1/2}^n (\rh{j}, \rh{j})  - F_{j-1/2}^n (\rh{j}, \rh{j}) }
    \leq
    C \, \dx \, \modulo{\rh{n}}
    \left( 1 + \mathcal{L} \, \mathcal{C}_1 (t^n) \right).
  \end{displaymath}
  Inserting the above estimate into~\eqref{eq:1} and exploiting the
  bounds on $\beta_j^n$ and $\gamma_j^n$ obtained in the proof of
  Lemma~\ref{lem:pos}, we get
  \begin{align*}
    \rho_j^{n+1} \leq \
    &
      (1 - \beta_j^n - \gamma_j^n) \, \rh{j} + \beta_j^n \, \rh{j-1} + \gamma_j^n \, \rh{j+1}
      + \lambda \modulo{F_{j+1/2}^n (\rh{j}, \rh{j})  - F_{j-1/2}^n (\rh{j}, \rh{j})}
    \\
    \leq \
    &  (1 - \beta_j^n - \gamma_j^n)  \norma{\rho^n}_{\L\infty (]a,b[)}
      + \beta_j^n \, \max\!\left\{\norma{\rho^n}_{\L\infty (]a,b[)}, \, \rho_a^n \right\}
      + \gamma_j^n \,  \max\!\left\{\norma{\rho^n}_{\L\infty (]a,b[)}, \,\rho_b^n\right\}
    \\
    & + \lambda  \, \dx \, C
      (1+\mathcal{L} \, \mathcal{C}_1 (t^n))
      \norma{\rho^n}_{\L\infty (]a,b[)}
    \\
    \leq \
    & \max\left\{\norma{\rho^n}_{\L\infty (]a,b[)}, \, \rho_a^n , \, \rho_b^n\right\}
      \left(1 + \dt \, C
      \left(1 + \mathcal{L} \, \mathcal{C}_1 (t^n)\right)\right)
    \\
    \leq \
    &
      e^{\mathcal{C}_2 (t^n) \dt}
      \max\left\{\norma{\rho^n}_{\L\infty (]a,b[)}, \, \rho_a^n , \, \rho_b^n\right\},
  \end{align*}
  where
  \begin{equation}
    \label{eq:C2}
    \mathcal{C}_2 (t) = C (1 + \mathcal{L} \, \mathcal{C}_1 (t)),
  \end{equation}
  $\mathcal{L}$ being as in~\eqref{eq:L}. An iterative argument,
  together with the fact that
  $\mathcal{C}_2 (t^{n-1}) \leq \mathcal{C}_2 (t^n)$ for all
  $n=1, \ldots, N_T$, yields the thesis.
\end{proof}

\subsection{BV estimates}
\label{sec:bv}

\begin{proposition}\label{prop:BV} {\bf ($\BV$ estimate in space)}
  Let $\rho_o \in \BV (\,]a,b[; \reali^+)$,
  $\rho_a, \, \rho_b \in \BV (\reali^+; \reali^+)$. Let~\ref{f},
  \ref{omega} and~\eqref{eq:CFL} hold.  Then, for all
  $n=1, \ldots, N_T$. the following estimate holds
  \begin{equation}
    \label{eq:spaceBV}
    \sum_{j=0}^{N}
    \modulo{\rho^{n}_{j+1}-\rho^{n}_{j}} \leq
    \mathcal{C}_x (t^n),
  \end{equation}
  where
  \begin{equation}
    \label{eq:11}
    \begin{aligned}
      \mathcal{C}_x (t^n) = \ & e^{\mathcal{K}_1 (t^n) \, t^n} \left[
        \sum_{j=0}^N \modulo{\rho_{j+1}^{0} -\rho_j^{0}} +
        \sum_{m=1}^n \modulo{\rho_{a}^{m} -\rho_a^{m-1}} +
        \sum_{m=1}^n \modulo{\rho_{b}^{m} -\rho_b^{m-1}} \right]
      \\
      & + \frac{\mathcal{K}_4 (t^n)}{\mathcal{K}_1 (t^n)} \, \left(
        e^{\mathcal{K}_1 (t^n) \, t^n} - 1 \right),
    \end{aligned}
  \end{equation}
  with $\mathcal{K}_1 (t^n)$ and $\mathcal{K}_4 (t^n)$ are defined in
  \eqref{eq:K123} and~\eqref{eq:K4}.
\end{proposition}

\begin{remark}
  Estimate~\eqref{eq:spaceBV} is defined also for $n=0$, setting
  $\sum_{m=1}^0 a_m = 0$, with some abuse of notation.
\end{remark}

\begin{proof}
  Consider the inner terms and the boundary ones separately.

  For $j=1, \ldots, N-1$ and $n =0, \ldots, N_T -1$, focus on the
  difference $ \rho^{n+1}_{j+1} - \rho^{n+1}_j$,
  exploiting~\eqref{eq:scheme}:
  \begin{align*}
    &\rho^{n+1}_{j+1} - \rho^{n+1}_j
    \\
    = \
    & \rh{j+1} - \rh{j}
    \\
    & - \lambda \left[
      F^n_{j+3/2} (\rh{j+1}, \rh{j+2}) -  F^n_{j+1/2} (\rh{j}, \rh{j+1})
      -  F^n_{j+1/2} (\rh{j}, \rh{j+1}) +  F^n_{j-1/2} (\rh{j-1}, \rh{j})
      \right]
    \\
    & \pm \lambda \, F^n_{j+3/2} (\rh{j}, \rh{j+1}) \pm \, \lambda F^n_{j+1/2} (\rh{j-1}, \rh{j})
    \\
    = \
    & \rh{j+1} - \rh{j}
    \\
    & - \lambda \left[
      F^n_{j+3/2} (\rh{j+1}, \rh{j+2})  - F^n_{j+1/2} (\rh{j}, \rh{j+1})
      + F^n_{j+1/2} (\rh{j-1}, \rh{j}) -F^n_{j+3/2} (\rh{j}, \rh{j+1})
      \right]
    \\
    & - \lambda \left[
      F^n_{j+3/2} (\rh{j}, \rh{j+1})  -F^n_{j+1/2} (\rh{j-1}, \rh{j})
      +F^n_{j-1/2} (\rh{j-1}, \rh{j}) - F^n_{j+1/2} (\rh{j}, \rh{j+1})
      \right]
    \\
    = \
    & \mathcal{A}_j^n - \lambda \, \mathcal{B}_j^n,
  \end{align*}
  where we set
  \begin{align*}
    \mathcal{A}_j^n = \
    &\rh{j+1} - \rh{j}
    \\
    & - \lambda \left[
      F^n_{j+3/2} (\rh{j+1}, \rh{j+2})  - F^n_{j+1/2} (\rh{j}, \rh{j+1})
      + F^n_{j+1/2} (\rh{j-1}, \rh{j}) -F^n_{j+3/2} (\rh{j}, \rh{j+1})
      \right],
    \\
    \mathcal{B}_j^n = \
    & F^n_{j+3/2} (\rh{j}, \rh{j+1})  -F^n_{j+1/2} (\rh{j-1}, \rh{j})
      +F^n_{j-1/2} (\rh{j-1}, \rh{j}) - F^n_{j+1/2} (\rh{j}, \rh{j+1}).
  \end{align*}
  Rearrange $\mathcal{A}_j^n$ as follows:
  \begin{align*}
    \mathcal{A}_j^n = \
    & \rh{j+1} - \rh{j}
      - \lambda \,
      \frac{F^n_{j+3/2} (\rh{j+1}, \rh{j+2}) - F^n_{j+3/2} (\rh{j+1}, \rh{j+1})}{\rh{j+2}- \rh{j+1}} \,
      \left(\rh{j+2}- \rh{j+1}\right)
    \\
    &- \lambda \,
      \frac{F^n_{j+3/2} (\rh{j+1}, \rh{j+1}) - F^n_{j+3/2} (\rh{j}, \rh{j+1})}{\rh{j+1}- \rh{j}} \,
      \left(\rh{j+1}- \rh{j}\right)
    \\
    & + \lambda \,
      \frac{F^n_{j+1/2} (\rh{j}, \rh{j+1}) - F^n_{j+1/2} (\rh{j}, \rh{j})}{\rh{j+1}- \rh{j}} \,
      \left(\rh{j+1}- \rh{j}\right)
    \\
    &   + \lambda \,
      \frac{F^n_{j+1/2} (\rh{j}, \rh{j}) - F^n_{j+1/2} (\rh{j-1}, \rh{j})}{\rh{j}- \rh{j-1}} \,
      \left(\rh{j}- \rh{j-1}\right)
    \\
    = \
    & \delta_j^n \, (\rh{j} - \rh{j-1}) + \gamma_{j+1}^n \, (\rh{j+2} - \rh{j+1})
      + (1 - \gamma_j^n - \delta_{j+1}^n) (\rh{j+1} - \rh{j}),
  \end{align*}
  where
  \begin{equation}
    \label{eq:deltajn}
    \delta_j^n =
    \begin{cases}
      \lambda \dfrac{F_{j+1/2}^n (\rh{j}, \rh{j})- F_{j+1/2}^n
        (\rh{j-1}, \rh{j})}{\rh{j}-\rh{j-1}} & \mbox{if } \rh{j} \neq
      \rh{j-1},
      \\
      0 & \mbox{if } \rh{j} = \rh{j-1},
    \end{cases}
  \end{equation}
  while $\gamma_j^n$ is as in~\eqref{eq:gammajn}. It can be proven
  that $\delta_j^n \in \left[0, 1/3 \right]$. Thus,
  \begin{align}
    \nonumber
    &  \sum_{j=1}^{N-1} \modulo{\mathcal{A}_j^n}
    \\ \nonumber
    \leq \
    &  \sum_{j=1}^{N-1} \modulo{\rh{j+1} - \rh{j}}
      + \sum_{j=0}^{N-2} \delta_{j+1}^n \, \modulo{\rh{j+1} - \rh{j}}
      - \sum_{j=1}^{N-1} \delta_{j+1}^n \, \modulo{\rh{j+1} - \rh{j}}
    \\ \nonumber
    &  + \sum_{j=2}^{N}\gamma_j^n \, \modulo{\rh{j+1} - \rh{j}}
      - \sum_{j=1}^{N-1} \gamma_j^n \, \modulo{\rh{j+1} - \rh{j}}
    \\ \label{eq:Ajn}
    = \
    & \sum_{j=1}^{N-1} \modulo{\rh{j+1} - \rh{j}}
      + \delta_1^n  \, \modulo{\rh{1} - \rh{a}} - \delta_N^n \, \modulo{\rh{N}- \rh{N-1}}
      + \gamma^n_N \, \modulo{\rh{b} - \rh{N}} - \gamma_1^n \, \modulo{\rh{2} - \rh{1}}.
  \end{align}
  Focus now on $\mathcal{B}_j^n$:
  \begin{align*}
    \mathcal{B}_j^n = \
    & \frac12 \left[
      f (t^n,x_{j+3/2}, \rh{j}, \Rh{j+3/2} ) + f (t^n,x_{j+3/2}, \rh{j+1}, \Rh{j+3/2} )
      \right.
    \\
    & - f (t^n,x_{j+1/2}, \rh{j}, \Rh{j+1/2} ) - f (t^n,x_{j+1/2}, \rh{j+1}, \Rh{j+1/2} )
    \\
    & + f (t^n,x_{j-1/2}, \rh{j-1}, \Rh{j-1/2} ) + f (t^n,x_{j-1/2}, \rh{j}, \Rh{j-1/2} )
    \\
    & \left.  - f (t^n,x_{j+1/2}, \rh{j-1}, \Rh{j+1/2} ) - f (t^n,x_{j+1/2}, \rh{j}, \Rh{j+1/2} )
      \right]
    \\
    = \
    & \frac12 \left[
      f (t^n,x_{j+3/2}, \rh{j+1}, \Rh{j+3/2} ) -  f (t^n,x_{j+1/2}, \rh{j+1}, \Rh{j+1/2} )
      \right]
    \\
    & + \frac12 \left[
      f (t^n,x_{j-1/2}, \rh{j-1}, \Rh{j-1/2} ) - f (t^n,x_{j+1/2}, \rh{j-1}, \Rh{j+1/2} )
      \right]
    \\
    & + \frac12 \left[
      f (t^n,x_{j+3/2}, \rh{j}, \Rh{j+3/2} ) -2 \,  f (t^n,x_{j+1/2}, \rh{j}, \Rh{j+1/2} )
      + f (t^n,x_{j-1/2}, \rh{j}, \Rh{j-1/2} )
      \right]
    \\
    = \
    & \frac12 \left[
      \partial_R f (t^n,x_{j+3/2}, \rh{j+1}, \tilde{R}^n_{j+1} ) \, (\Rh{j+3/2} - \Rh{j+1/2})
      + \dx \, \partial_x f (t^n,\tilde{x}_{j+1}, \rh{j+1}, \Rh{j+1/2} )
      \right]
    \\
    & -  \frac12 \left[
      \partial_R f (t^n,x_{j-1/2}, \rh{j-1}, \tilde{R}^n_{j} ) \, (\Rh{j+1/2} - \Rh{j-1/2})
      + \dx \, \partial_x f (t^n,\tilde{x}_{j}, \rh{j-1}, \Rh{j+1/2} )
      \right]
    \\
    & + \frac12 \left[
      f (t^n,x_{j+3/2}, \rh{j}, \Rh{j+1/2} ) -2 \,  f (t^n,x_{j+1/2}, \rh{j}, \Rh{j+1/2} )
      + f (t^n,x_{j-1/2}, \rh{j}, \Rh{j+1/2} )
      \right]
    \\
    & + \frac12 \left[
      f (t^n,x_{j+3/2}, \rh{j}, \Rh{j+3/2} )  -  f (t^n,x_{j+3/2}, \rh{j}, \Rh{j+1/2} )
      \right]
    \\
    & + \frac12\left[
      f (t^n,x_{j-1/2}, \rh{j}, \Rh{j-1/2}) -  f (t^n,x_{j-1/2}, \rh{j}, \Rh{j+1/2})
      \right]
    \\
    = \
    & \frac{\dx}{2}\left[
      (\tilde{x}_{j+1} - \tilde{x}_j) \, \partial_{xx}^2 f (t^n,\hat{x}_{j+1/2}, \rh{j+1}, \Rh{j+1/2} )
      + \partial_{\rho x}^2  f (t^n,\tilde{x}_{j}, \tilde{\rho}^n_{j}, \Rh{j+1/2} ) \, (\rh{j+1} - \rh{j-1})
      \right]
    \\
    & + \frac12 \left[
      2 \, \dx \, \partial_{xR}^2 f (t^n,\check{x}_{j+1/2}, \rh{j+1}, \tilde{R}^n_{j+1} )
      \, (\Rh{j+3/2} - \Rh{j+1/2}) \right.
    \\
    & \qquad\quad
      +  \partial_{\rho R}^2 f (t^n,x_{j-1/2}, \bar{\rho}^n_{j}, \tilde{R}^n_{j+1} )
      \, (\Rh{j+3/2} - \Rh{j+1/2}) \, (\rh{j+1}- \rh{j-1})
    \\
    & \qquad\quad
      + \partial_{RR}^2 f (t^n,x_{j-1/2}, \rh{j-1}, \hat{R}^n_{j+1/2} )
      \, (\Rh{j+3/2} - \Rh{j+1/2}) \, ( \tilde{R}^n_{j+1} - \tilde{R}^n_{j})
    \\
    & \qquad\quad\left.
      +  \partial_R f (t^n,x_{j-1/2}, \rh{j-1}, \tilde{R}^n_{j} )
      \, (\Rh{j+3/2} - \Rh{j+1/2} - \Rh{j+1/2} + \Rh{j-1/2})
      \right]
    \\
    & + \frac{\dx}{2}\left[
      \partial_x f (t^n, \bar{x}_{j+1}, \rh{j}, \Rh{j+1/2})
      -  \partial_x f (t^n, \bar{x}_{j}, \rh{j}, \Rh{j+1/2})
      \right]
    \\
    & + \frac12 \left[
      \partial_R f (t^n, x_{j+3/2}, \rh{j}, \bar{R}^n_{j+1}) \, (\Rh{j+3/2}- \Rh{j+1/2})\right.
    \\
    & \quad\qquad \left.
      -  \partial_R f (t^n, x_{j-1/2}, \rh{j}, \bar{R}^n_{j}) \, (\Rh{j+1/2}- \Rh{j-1/2})
      \right]
    \\
    = \
    &  \frac{\dx}{2}\left[
      (\tilde{x}_{j+1} - \tilde{x}_j) \, \partial_{xx}^2 f (t^n,\hat{x}_{j+1/2}, \rh{j+1}, \Rh{j+1/2} )
      + \partial_{\rho x}^2  f (t^n,\tilde{x}_{j}, \tilde{\rho}^n_{j}, \Rh{j+1/2} ) \, (\rh{j+1} - \rh{j-1})
      \right]
    \\
    & + \frac12 \left[
      2 \, \dx \, \partial_{xR}^2 f (t^n,\check{x}_{j+1/2}, \rh{j+1}, \tilde{R}^n_{j+1} )
      \, (\Rh{j+3/2} - \Rh{j+1/2}) \right.
    \\
    & \qquad\quad
      +  \partial_{\rho R}^2 f (t^n,x_{j-1/2}, \bar{\rho}^n_{j}, \tilde{R}^n_{j+1} )
      \, (\Rh{j+3/2} - \Rh{j+1/2}) \, (\rh{j+1}- \rh{j-1})
    \\
    & \qquad\quad
      + \partial_{RR}^2 f (t^n,x_{j-1/2}, \rh{j-1}, \hat{R}^n_{j+1/2} )
      \, (\Rh{j+3/2} - \Rh{j+1/2}) \, ( \tilde{R}^n_{j+1} - \tilde{R}^n_{j})
    \\
    & \qquad\quad\left.
      +  \partial_R f (t^n,x_{j-1/2}, \rh{j-1}, \tilde{R}^n_{j} )
      \, (\Rh{j+3/2} - \Rh{j+1/2} - \Rh{j+1/2} + \Rh{j-1/2})
      \right]
    \\
    & + \frac{\dx}{2} \,  (\bar{x}_{j+1} - \bar{x}_j) \,
      \partial_{xx}^2  f (t^n, \bar{x}_{j+1/2}, \rh{j}, \Rh{j+1/2})
    \\
    & + \frac12\left[
      2 \, \dx \, \partial_{xR}^2 f (t^n,\overline{\overline{x}}_{j+1/2}, \rh{j}, \bar{R}^n_{j+1} )
      \, (\Rh{j+3/2} - \Rh{j+1/2}) \right.
    \\
    & \qquad\quad
      +  \partial_{RR}^2 f (t^n,x_{j-1/2}, \rh{j}, \bar{\bar{R}}^n_{j+1/2} )
      \, (\Rh{j+3/2} - \Rh{j+1/2}) \, ( \bar{R}^n_{j+1} - \bar{R}^n_{j})
    \\
    &\qquad \quad \left.
      +  \partial_R f (t^n,x_{j-1/2}, \rh{j}, \bar{R}^n_{j} )
      \, (\Rh{j+3/2} - \Rh{j+1/2} - \Rh{j+1/2} + \Rh{j-1/2})
      \right],
  \end{align*}
  where
  \begin{align*}
    \tilde{R}^n_{j+1}, \bar{R}^n_{j+1} \in
    &  \mathcal{I}\left(\Rh{j+1/2}, \Rh{j+3/2}\right),
    &
      \tilde{x}_{j+1}, \bar{x}_{j+1} \in
    & \, ]x_{j+1/2}, x_{j+3/2}[,
    \\
    \tilde{R}^n_{j}, \bar{R}^n_{j} \in
    & \mathcal{I}\left(\Rh{j-1/2}, \Rh{j+1/2}\right),
    &
      \tilde{x}_{j}, \bar{x}_{j} \in
    & \,]x_{j-1/2}, x_{j+1/2}[,
    \\
    \hat{R}^n_{j+1/2} \in
    & \mathcal{I}\left(\tilde{R}^n_{j}, \tilde{R}^n_{j+1}\right),
    &
      \hat{x}_{j+1/2} \in
    &\, ]\tilde{x}_{j}, \tilde{x}_{j+1}[,
    \\
    \bar{\bar{R}}^n_{j+1/2} \in
    &  \mathcal{I}\left( \bar{R}^n_{j}, \bar{R}^n_{j+1}\right),
    &
      \check{x}_{j+1/2}, \overline{\overline{x}}_{j+1/2} \in
    &\, ]x_{j-1/2},  x_{j+3/2}[,
    \\
    \tilde{\rho}^n_{j}, \bar{\rho}^n_{j} \in
    & \mathcal{I}\left(\rh{j-1}, \rh{j+1}\right),
    &
      \bar{x}_{j+1/2} \in
    &\, ]\bar{x}_j, \bar{x}_{j+1}[.
  \end{align*}
  Notice that
  \begin{align*}
    \modulo{\tilde{x}_{j+1} - \tilde{x}_j} \leq \
    2 \, \dx,
    \quad
    \modulo{\bar{x}_{j+1} - \bar{x}_j} \leq \
    2 \, \dx,
    \quad
    \modulo{\Rh{j+3/2} - \Rh{j+1/2}} \leq \
    \mathcal{L} \, \mathcal{C}_1 (t^n) \, \dx,
  \end{align*}
  where $\mathcal{L}$ is as in~\eqref{eq:L}.  Moreover, by their very
  definition, for $\theta^n_{j+1}, \epsilon^n_{j} \in [0,1]$,
  \begin{align*}
    \modulo{ \tilde{R}^n_{j+1} - \tilde{R}^n_{j}} = \
    & \modulo{\theta^n_{j+1} \, \Rh{j+3/2} + (1-\theta^n_{j+1}) \, \Rh{j+1/2}
      - \epsilon^n_{j} \, \Rh{j+1/2} - (1-\epsilon^n_{j}) \, \Rh{j-1/2}}
    \\
    \leq \
    & \modulo{\Rh{j+3/2} - \Rh{j+1/2}} +  \modulo{\Rh{j+1/2} - \Rh{j-1/2}}
    \\
    \leq \
    &2\, \mathcal{L} \, \mathcal{C}_1 (t^n) \, \dx,
  \end{align*}
  and similarly
  $\modulo{ \bar{R}^n_{j+1} - \bar{R}^n_{j}} \leq 2 \, \mathcal{L} \,
  \mathcal{C}_1 (t^n) \, \dx$. Compute now
  \begin{equation}
    \label{eq:diffdiff}
    \!\! \Rh{j+3/2} - 2 \,  \Rh{j+1/2}+ \Rh{j-1/2}
    = \dx \! \left[
      \sum_{k=1}^N \rh{k} \left(
        \frac{\omega^{k-j-1}}{W_{j+3/2}}
        -  \frac{\omega^{k-j}}{W_{j+1/2}} -  \frac{\omega^{k-j}}{W_{j+1/2}}
        + \frac{\omega^{k-j+1}}{W_{j-1/2}}
      \right) \!
    \right].
  \end{equation}
  Observe that, for $k$ fixed:
  \begin{align}
    \nonumber
    &\frac{\omega^{k-j-1}}{W_{j+3/2}}  -  \frac{\omega^{k-j}}{W_{j+1/2}}
    \\ \nonumber
    = \
    & \frac{\dx}{W_{j+3/2} \, W_{j+1/2}} \left[\sum_{\ell=1}^N
      \left(
      \omega^{k-j-1} \, \omega^{\ell-j} - \omega^{k-j} \, \omega^{\ell - j -1}
      \right)
      \right]
    \\ \nonumber
    = \
    & \frac{\dx}{W_{j+3/2} \, W_{j+1/2}} \left[\sum_{\ell=1}^N
      \left(
      \omega^{k-j-1} \left(\omega^{\ell-j} - \omega^{\ell - j -1}\right)
      + \left(\omega^{k-j-1} - \omega^{k-j} \right) \omega^{\ell - j -1}
      \right)
      \right]
    \\ \label{eq:2}
    = \
    &  \frac{(\dx)^2}{W_{j+3/2} \, W_{j+1/2}} \left[\sum_{\ell=1}^N
      \left(
      \omega^{k-j-1}\, \omega'(\xi_{\ell-j-1/2})
      -  \omega^{\ell - j -1} \, \omega' (\xi_{k-j-1/2})
      \right)
      \right],
  \end{align}
  where $\xi_{\ell-j-1/2} \in\, ]y_{\ell-j-1}, y_{\ell-j}[$,
  $\xi_{k-j-1/2} \in\, ]y_{k-j-1}, y_{k-j}[$, and similarly
  \begin{equation}
    \label{eq:3}
    \frac{\omega^{k-j}}{W_{j+1/2}}  -  \frac{\omega^{k-j+1}}{W_{j-1/2}}
    =
    \frac{(\dx)^2}{W_{j+1/2} \, W_{j-1/2}}  \left[\sum_{\ell=1}^N
      \left(
        \omega^{k-j}\, \omega'(\xi_{\ell-j+1/2})
        -  \omega^{\ell - j } \, \omega' (\xi_{k-j+1/2})
      \right)
    \right].
  \end{equation}
  In order to compute the difference between~\eqref{eq:2}
  and~\eqref{eq:3}, which appears in~\eqref{eq:diffdiff}, we add and
  subtract the term
  \begin{displaymath}
    \frac{(\dx)^2}{W_{j+3/2} \, W_{j+1/2}} \left[\sum_{\ell=1}^N
      \left( \omega^{k-j}\, \omega'(\xi_{\ell-j+1/2}) - \omega^{\ell - j} \, \omega' (\xi_{k-j+1/2})
      \right) \right].
  \end{displaymath}
  The terms with common denominator yield
  \begin{align}
    \nonumber
    & \frac{(\dx)^2}{W_{j+3/2} \, W_{j+1/2}}
      \left[\sum_{\ell=1}^N
      \left(
      \omega^{k-j-1}\, \omega'(\xi_{\ell-j-1/2})
      -  \omega^{\ell - j -1} \, \omega' (\xi_{k-j-1/2})
      \right. \right.
    \\ \nonumber
    & \quad\qquad\qquad\qquad\qquad \left.-
      \omega^{k-j}\, \omega'(\xi_{\ell-j+1/2})
      +  \omega^{\ell - j } \, \omega' (\xi_{k-j+1/2})
      \right)\Biggr]
    \\ \nonumber
    = \
    &  \frac{(\dx)^2}{W_{j+3/2} \, W_{j+1/2}}
      \left[\sum_{\ell=1}^N
      \left(
      (\omega^{k-j-1} - \omega^{k-j}) \,  \omega'(\xi_{\ell-j-1/2})
      +  \omega^{k-j} \, (\omega'(\xi_{\ell-j-1/2}) - \omega'(\xi_{\ell-j+1/2}))
      \right.\right.
    \\ \nonumber
    & \qquad\qquad
      \left.   -  (\omega^{\ell - j -1} - \omega^{\ell-j}) \, \omega' (\xi_{k-j-1/2})
      -  \omega^{\ell - j } \, (\omega' (\xi_{k-j-1/2}) - \omega' (\xi_{k-j+1/2}))
      \right)\Biggr]
    \\ \nonumber
    = \
    & \frac{(\dx)^2}{W_{j+3/2} \, W_{j+1/2}}
      \left[\sum_{\ell=1}^N
      \left(-\dx \, \omega' (\xi_{k-j-1/2}) \,  \omega'(\xi_{\ell-j-1/2})
      -   \omega^{k-j} \int_{\xi_{\ell-j-1/2}}^{\xi_{\ell-j+1/2}} \omega'' (y) \d{y}
      \right. \right.
    \\ \nonumber
    & \qquad\qquad\qquad\qquad
      \left.
      + \dx \, \omega' (\xi_{\ell-j-1/2}) \, \omega' (\xi_{k-j-1/2})
      +   \omega^{\ell - j } \int_{\xi_{k-j-1/2}}^{\xi_{k-j+1/2}} \omega'' (y) \d{y}
      \right)\Biggr]
    \\ \label{eq:4}
    = \
    & \frac{(\dx)^2}{W_{j+3/2} \, W_{j+1/2}}
      \left[\sum_{\ell=1}^N\left(
      \omega^{\ell - j } \int_{\xi_{k-j-1/2}}^{\xi_{k-j+1/2}} \omega'' (y) \d{y}
      - \omega^{k-j} \int_{\xi_{\ell-j-1/2}}^{\xi_{\ell-j+1/2}} \omega'' (y) \d{y}
      \right)\right].
  \end{align}
  We are left with
  \begin{equation}
    \label{eq:6}
    \!\!\left[\sum_{\ell=1}^N
      \left( \omega^{k-j}\, \omega'(\xi_{\ell-j+1/2}) - \omega^{\ell - j} \, \omega' (\xi_{k-j+1/2})
      \right) \right] \!\left(
      \frac{(\dx)^2}{W_{j+3/2} \, W_{j+1/2}} -     \frac{(\dx)^2}{W_{j+1/2} \, W_{j-1/2}}
    \right).
  \end{equation}
  In particular, observe that
  \begin{align}
    \nonumber
    \frac{1}{W_{j+3/2} \, W_{j+1/2}} -     \frac{1}{W_{j+1/2} \, W_{j-1/2}} = \
    & \frac{ W_{j-1/2} - W_{j+3/2}}{W_{j+3/2} \, W_{j+1/2} \, W_{j-1/2}}
    \\ \nonumber
    = \
    & \frac{\dx}{W_{j+3/2} \, W_{j+1/2} \, W_{j-1/2}}
      \left(\sum_{\beta=1}^N\left(
      \omega^{\beta-j+1} - \omega^{\beta-j-1}
      \right)\right)
    \\ \nonumber
    = \
    &  \frac{\dx}{W_{j+3/2} \, W_{j+1/2} \, W_{j-1/2}}
      \left(\sum_{\beta=1}^N \int_{y_{\beta-j-1}}^{y_{\beta-j+1}}
      \omega' (y) \d{y}
      \right)
    \\  \label{eq:6b}
    \leq \
    & \frac{2 \, \dx \, \norma{\omega'}_{\L1}}{W_{j+3/2} \, W_{j+1/2} \, W_{j-1/2}}.
  \end{align}
  Coming back to~\eqref{eq:diffdiff}, exploiting \eqref{eq:4},
  \eqref{eq:6} and~\eqref{eq:6b}, we get
  \begin{align}
    \nonumber
    & \modulo{\Rh{j+3/2} - 2 \,  \Rh{j+1/2}+ \Rh{j-1/2}}
    \\ \nonumber
    \leq \
    &\frac{(\dx)^3}{\modulo{W_{j+3/2} \, W_{j+1/2}}} \left|
      \left(\sum_{k=1}^N \rh{k} \,  \int_{\xi_{k-j-1/2}}^{\xi_{k-j+1/2}} \omega'' (y) \d{y}\right) \!
      \left( \sum_{\ell=1}^N \omega^{\ell - j }\right)\right.
    \\ \nonumber
    & \qquad\qquad\qquad \qquad\left.
      - \left(\sum_{k=1}^N \rh{k} \, \omega^{k-j}
      \right)\!
      \left( \sum_{\ell=1}^N \int_{\xi_{\ell-j-1/2}}^{\zeta_{\ell-j+1/2}} \omega'' (y) \d{y} \right)\right|
    \\ \nonumber
    & + \frac{2 \, (\dx)^4 \, \norma{\omega'}_{\L1}}
      {\modulo{W_{j+3/2} \, W_{j+1/2} \, W_{j-1/2}}}
      \left|
      \left(\sum_{k=1}^N \rh{k} \, \omega^{k-j} \right)
      \left( \sum_{\ell=1}^N \omega' (\xi_{\ell-j+1/2}) \right)
      \right.
    \\ \nonumber
    & \qquad\qquad\qquad \qquad\qquad\qquad\qquad\qquad\left.
      -  \left(\sum_{k=1}^N \rh{k} \, \omega'(\xi_{k-j+1/2}) \right)
      \left( \sum_{\ell=1}^N \omega^{\ell-j} \right)
      \right|
    \\ \nonumber
    \leq \
    & \frac{2 \,(\dx)^2}{K_\omega} \,  \norma{\omega''}_{\L\infty} \norma{\rho^n}_{\L1}
      +
      \frac{(\dx)^2}{K_\omega^2} \, \norma{\omega}_{\L\infty}  \norma{\omega''}_{\L1} \norma{\rho^n}_{\L1}
    \\ \nonumber
    & + \frac{2 \, (\dx)^2}{K_\omega^3} \norma{\omega}_{\L\infty} \norma{\omega'}^2_{\L1}
      \norma{\rho^n}_{\L1}
      +
      \frac{2 \, (\dx)^2}{K_\omega^2} \norma{\omega'}_{\L\infty} \norma{\omega'}_{\L1}
      \norma{\rho^n}_{\L1}
    \\ \label{eq:diffdiffOK}
    = \
    &\dx^2 \, \mathcal{W} \,  \norma{\rho^n}_{\L1},
  \end{align}
  where we set
  \begin{equation}
    \label{eq:Wconst}
    \mathcal{W} =
    \frac{2}{K_\omega} \, \norma{\omega''}_{\L\infty}
    +
    \frac{1}{K_\omega^2} \, \norma{\omega}_{\L\infty} \norma{\omega''}_{\L1}
    +  \frac{2}{K_\omega^3} \, \norma{\omega}_{\L\infty} \norma{\omega'}^2_{\L1}
    +
    \frac{2}{K_\omega^2} \, \norma{\omega'}_{\L\infty} \norma{\omega'}_{\L1}.
  \end{equation}
  Hence,
  \begin{align*}
    \modulo{\mathcal{B}^n_j} \leq \
    & (\dx)^2 \, C \, \modulo{\rh{j+1}}
      +
      \frac{\dx}2 \, \norma{\partial_{\rho x}^2 f}_{\L\infty} \modulo{\rh{j+1} - \rh{j-1}}
      +
      (\dx) ^2 \, C \,  \mathcal{L} \, \mathcal{C}_1 (t^n) \, \modulo{\rh{j+1}}
    \\
    & +
      \frac{\dx}{2} \, \mathcal{L}\, \mathcal{C}_1 (t^n) \,
      \norma{\partial_{\rho R}^2 f}_{\L\infty} \modulo{\rh{j+1} - \rh{j-1}}
      +
      (\dx)^2 \, C \, \mathcal{L}^2\, (\mathcal{C}_1 (t^n))^2 \, \modulo{\rh{j-1}}
      +
      (\dx)^2 \, C \, \modulo{\rh{j}}
    \\
    & +
      (\dx)^2 \, C \, \mathcal{L}\, \mathcal{C}_1 (t^n) \, \modulo{\rh{j}}
      +
      (\dx)^2 \, C \, \mathcal{L}^2\, (\mathcal{C}_1 (t^n))^2 \, \modulo{\rh{j}}
    \\
    & +  \frac12 \,
      (\dx)^2 \, C \, \mathcal{W} \, \mathcal{C}_1 (t^n) \,
      \left(\modulo{\rh{j-1}} + \modulo{\rh{j}}\right).
  \end{align*}
  Therefore,
  \begin{align}
    \nonumber
    \sum_{j=1}^{N-1} \lambda \, \modulo{\mathcal{B}_j^n}
    \leq \
    & \frac{\dt}{2} \, \left(
      \norma{\partial_{\rho x}^2 f}_{\L\infty}
      +  \mathcal{L} \, \mathcal{C}_1 (t^n) \, \norma{\partial_{\rho R}^2 f}_{\L\infty}
      \right)
      \sum_{j=1}^{N-1}  \modulo{\rh{j+1} - \rh{j-1}}
    \\
    \nonumber
    & + \dt \, C \left(
      1 + 2 \, \mathcal{L} \, \mathcal{C}_1 (t^n)
      + 2 \, \mathcal{L}^2 \, (\mathcal{C}_1 (t^n))^2
      +  \mathcal{W} \, \mathcal{C}_1 (t^n)
      \right) \, \dx \, \sum_{j=1}^N \modulo{\rh{j}}
    \\
    \nonumber
    & + \dt \, \dx \, C \left( \mathcal{L}^2 \, (\mathcal{C}_1 (t^n))^2
      + \frac{1}{2} \, \mathcal{W} \, \mathcal{C}_1 (t^n)\right) \modulo{\rh{a}}
    \\  \label{eq:Bjn}
    \leq \
    & \mathcal{K}_1 (t^n) \, \dt
      \left(\sum_{j=1}^{N-1} \modulo{\rh{j+1} - \rh{j}}
      + \frac12 \, \modulo{\rh{1} - \rh{a}}\right)
      + \mathcal{K}_2 (t^n) \, \dt \,
      + \mathcal{K}_3 (t^n) \, \dt \, \modulo{\rh{a}},
  \end{align}
  with
  \begin{align}
    \nonumber
    \mathcal{K}_1 (t) = \
    &  \norma{\partial_{\rho x}^2 f}_{\L\infty}
      +  \mathcal{L}\, \mathcal{C}_1 (t) \, \norma{\partial_{\rho R}^2 f}_{\L\infty},
    \\[3pt]
    \label{eq:K123}
    \mathcal{K}_2 (t) = \
    & C \, \mathcal{C}_1 (t) \, \left(
      1 + 2 \, \mathcal{L}\, \mathcal{C}_1 (t)  +2 \, \mathcal{K}_3 (t)
      \right),
    \\ \nonumber
    \mathcal{K}_3 (t) = \
    & C \, \mathcal{C}_1 (t) \left(\mathcal{L}^2 \, \mathcal{C}_1 (t)
      + \frac{1}{2} \,  \mathcal{W} \, \right),
  \end{align}
  $\mathcal{C}_1 (t)$ is as in~\eqref{eq:C1} and $\mathcal{L}$ is as
  in~\eqref{eq:L}. Observe that the two norms of $f$ appearing in
  $\mathcal{K}_1(t)$ are bounded due to~\ref{f}, Lemma~\ref{lem:l1}
  and Lemma~\ref{lem:linf}, since they are evaluated on the compact
  set
  $[0,t] \times [a,b] \times [-\norma{\rho(t)}_{\L\infty},
  \norma{\rho(t)}_{\L\infty}] \times [-J(t), J(t)]$, with
  $J(t) = \dfrac{\norma{\omega}_{\L\infty}}{K_\omega} \,
  \mathcal{C}_1(t)$.

  Focus now on the boundary terms. From the definition of the
  scheme~\eqref{eq:scheme}, with the
  notation~\eqref{eq:betajn}--\eqref{eq:gammajn}
  and~\eqref{eq:deltajn}, we have
  \begin{align*}
    & \rho_1^{n+1} - \rho_a^{n+1}
    \\
    = \
    & (1 - \beta_1^n - \gamma_1^n) \rh{1} + \beta_1^n \, \rh{a} + \gamma_1^n \, \rh{2}
      - \lambda \left[
      F_{3/2}^n (\rh{1}, \rh{1}) - F_{1/2}^n (\rh{1}, \rh{1})
      \right] - \rho_a^{n+1} \pm \rh{a}
    \\
    =\
    & \gamma_1^n (\rh{2} - \rh{1}) + (1-\beta_1^n) (\rh{1}- \rh{a}) + (\rho_a^{n+1} - \rh{a})
      - \lambda \left[
      F_{3/2}^n (\rh{1}, \rh{1}) - F_{1/2}^n (\rh{1}, \rh{1})
      \right]
    \\
    = \
    & \gamma_1^n (\rh{2} - \rh{1}) + (1-\delta_1^n) (\rh{1}- \rh{a}) + (\rho_a^{n+1} - \rh{a})
      - \lambda \left[
      F_{3/2}^n (\rh{a}, \rh{1}) - F_{1/2}^n (\rh{a}, \rh{1})
      \right],
  \end{align*}
  since
  \begin{align*}
    &\beta_1^n \,  (\rh{1}- \rh{a})
      +  \lambda \left[
      F_{3/2}^n (\rh{1}, \rh{1}) - F_{1/2}^n (\rh{1}, \rh{1})
      \right]
    \\
    = \
    & \lambda \left[
      F_{1/2}^n (\rh{1}, \rh{1}) - F_{1/2}^n (\rh{a}, \rh{1})
      + F_{3/2}^n (\rh{1}, \rh{1}) - F_{1/2}^n (\rh{1}, \rh{1})
      \pm F_{3/2}^n (\rh{a}, \rh{1})
      \right]
    \\
    = \
    & \delta_1^n \, (\rh{1} - \rh{a}) + \lambda \left[
      F_{3/2}^n (\rh{a}, \rh{1}) - F_{1/2}^n (\rh{a}, \rh{1})
      \right].
  \end{align*}
  Observing that
  \begin{align*}
    & \lambda \left[
      F_{3/2}^n (\rh{a}, \rh{1}) - F_{1/2}^n (\rh{a}, \rh{1})
      \right]
    \\
    = \
    & \frac{\lambda}{2}\left[
      f (t^n, x_{3/2}, \rh{a}, \Rh{3/2}) + f (t^n, x_{3/2}, \rh{1}, \Rh{3/2})
      - f (t^n, x_{1/2}, \rh{a}, \Rh{1/2}) - f (t^n, x_{1/2}, \rh{1}, \Rh{1/2})
      \right]
    \\
    = \
    & \frac{\lambda}{2}\left[
      \partial_x f (t^n, \tilde{x}_{1}, \rh{a}, \Rh{3/2}) \, \dx
      + \partial_R f (t^n, x_{1/2}, \rh{a}, \tilde{R}^n_{1}) \, (\Rh{3/2} - \Rh{1/2})
      \right.
    \\
    & \qquad\left.
      + \partial_x f (t^n, \overline{x}_{1}, \rh{1}, \Rh{3/2}) \, \dx
      + \partial_R f (t^n, x_{1/2}, \rh{1}, \bar{R}^n_{1}) \, (\Rh{3/2} - \Rh{1/2})
      \right],
  \end{align*}
  where $\tilde{x}_1, \overline{x}_1 \in\, ]x_{1/2}, x_{3/2}[$ and
  $\tilde{R}^n_1, \bar{R}^n_1 \in \mathcal{I} (\Rh{1/2}, \Rh{3/2})$,
  we conclude
  \begin{displaymath}
    \lambda\, \modulo{ F_{3/2}^n (\rh{a}, \rh{1}) - F_{1/2}^n (\rh{a}, \rh{1})  }
    \leq
    \dt \, \frac{C}{2} \, (1 + \mathcal{L} \, \mathcal{C}_1 (t^n)) \, (|\rh{a}| + |\rh{1}|).
  \end{displaymath}
  By the positivity of the coefficients involved, we obtain
  \begin{equation}
    \label{eq:diff1a}
    \begin{aligned}
      \modulo{\rho_1^{n+1} - \rho_a^{n+1}} \leq \ & \gamma_1^n \,
      |\rh{2} - \rh{1}| + (1-\delta_1^n) \, |\rh{1}- \rh{a}| +
      |\rho_a^{n+1} - \rh{a}|
      \\
      & + \dt \, \frac{C}{2} \, (1 + \mathcal{L}\, \mathcal{C}_1
      (t^n)) \, (|\rh{a}| + |\rh{1}|).
    \end{aligned}
  \end{equation}
  Concerning the other boundary term, we have
  \begin{align*}
    \rho_b^{n+1} - \rho_N^{n+1} = \
    &  \rho_b^{n+1} \pm \rh{b}
      - (1 - \beta_N^n - \gamma_N^n) \rh{N} + \beta_N^n \, \rh{N-1} + \gamma_N^n \, \rh{b}
    \\
    & + \lambda \left[
      F_{N+1/2}^n (\rh{N}, \rh{N}) - F_{N-1/2}^n (\rh{N}, \rh{N})
      \right]
    \\ =\
    & (\rho_b^{n+1} - \rh{b}) +
      (1 - \gamma_N^n) (\rh{b} - \rh{N}) + \beta_N^n (\rh{N}- \rh{N-1})
    \\
    &+ \lambda \left[
      F_{N+1/2}^n (\rh{N}, \rh{N}) - F_{N-1/2}^n (\rh{N}, \rh{N})
      \right]
    \\
    = \
    &   (\rho_b^{n+1} - \rh{b}) +
      (1 - \gamma_N^n) (\rh{b} - \rh{N}) + \delta_N^n (\rh{N}- \rh{N-1})
    \\
    &+ \lambda \left[
      F_{N+1/2}^n (\rh{N-1}, \rh{N}) - F_{N-1/2}^n (\rh{N-1}, \rh{N})
      \right],
  \end{align*}
  since
  \begin{align*}
    &\beta_N^n (\rh{N}- \rh{N-1})
      + \lambda \left[
      F_{N+1/2}^n (\rh{N}, \rh{N}) - F_{N-1/2}^n (\rh{N}, \rh{N})
      \right]
    \\
    = \
    & \lambda \left[
      F_{N-1/2}^n (\rh{N}, \rh{N}) - F_{N-1/2}^n (\rh{N-1}, \rh{N})
      + F_{N+1/2}^n (\rh{N}, \rh{N})\right.
    \\
    &\qquad\qquad \left.- F_{N-1/2}^n (\rh{N}, \rh{N})
      \pm F_{N+1/2}^n (\rh{N-1}, \rh{N})
      \right]
    \\
    = \
    &\delta_N^n (\rh{N}- \rh{N-1})
      + \lambda \left[
      F_{N+1/2}^n (\rh{N-1}, \rh{N}) - F_{N-1/2}^n (\rh{N-1}, \rh{N})
      \right].
  \end{align*}
  Observing that
  \begin{align*}
    & \lambda \left[
      F_{N+1/2}^n (\rh{N-1}, \rh{N}) - F_{N-1/2}^n (\rh{N-1}, \rh{N})
      \right]
    \\
    = \
    & \frac{\lambda}{2}\left[
      f (t^n, x_{N+1/2}, \rh{N-1}, \Rh{N+1/2}) + f (t^n, x_{N+1/2}, \rh{N}, \Rh{N+1/2})
      \right.
    \\
    & \quad \left.
      - f (t^n, x_{N-1/2}, \rh{N-1}, \Rh{N-1/2}) - f (t^n, x_{N-1/2}, \rh{N}, \Rh{N-1/2})
      \right]
    \\
    = \
    & \frac{\lambda}{2}\left[
      \partial_x f (t^n, \tilde{x}_{N}, \rh{N-1}, \Rh{N+1/2}) \, \dx
      + \partial_R f (t^n, x_{N-1/2}, \rh{N-1}, \tilde{R}^n_{N}) \, (\Rh{N+1/2} - \Rh{N-1/2})
      \right.
    \\
    & \qquad\left.
      + \partial_x f (t^n, \overline{x}_{N}, \rh{N}, \Rh{N+1/2}) \, \dx
      + \partial_R f (t^n, x_{N-1/2}, \rh{N}, \bar{R}^n_{N}) \, (\Rh{N+1/2} - \Rh{N-1/2})
      \right],
  \end{align*}
  where $\tilde{x}_N, \overline{x}_N \in \, ]x_{N-1/2}, x_{N+1/2}[$
  and
  $\tilde{R}^n_N, \bar{R}^n_N \in \mathcal{I} (\Rh{N-1/2},
  \Rh{N+1/2})$, we conclude
  \begin{displaymath}
    \lambda \, \modulo{ F_{N+1/2}^n (\rh{N-1}, \rh{N}) - F_{N-1/2}^n (\rh{N-1}, \rh{N}) }
    \leq
    \dt \, \frac{C}{2} \, (1 + \mathcal{L}\, \mathcal{C}_1 (t^n)) \, (|\rh{N-1}| + |\rh{N}|).
  \end{displaymath}
  By the positivity of the coefficients involved, we obtain
  \begin{equation}
    \label{eq:diffNb}
    \begin{aligned}
      \modulo{\rho_b^{n+1} - \rho_N^{n+1}} \leq \ & |\rho_b^{n+1} -
      \rh{b}| + (1 - \gamma_N^n) \, |\rh{b} - \rh{N}| + \delta_N^n \,
      |\rh{N}- \rh{N-1}|
      \\
      & + \dt \, \frac{C}{2} \, (1 + \mathcal{L}\, \mathcal{C}_1
      (t^n)) \, (|\rh{N-1}| + |\rh{N}|).
    \end{aligned}
  \end{equation}

  Collect now the estimates~\eqref{eq:Ajn}, \eqref{eq:Bjn},
  \eqref{eq:diff1a} and~\eqref{eq:diffNb}:
  \begin{align*}
    & \sum_{j=0}^N \modulo{\rho_{j+1}^{n+1}  -\rho_j^{n+1}}
    \\
    = \
    & \modulo{\rho_1^{n+1} - \rho_a^{n+1}}
      + \sum_{j=1}^{N-1}  \modulo{\rho_{j+1}^{n+1}  -\rho_j^{n+1}}
      +  \modulo{\rho_b^{n+1} - \rho_n^{n+1}}
    \\
    \leq \
    &
      \gamma_1^n |\rh{2} - \rh{1}| + (1-\delta_1^n) |\rh{1}- \rh{a}| + |\rho_a^{n+1} -   \rh{a}|
      + \dt \, \frac{C}{2} \, (1 + \mathcal{L}\, \mathcal{C}_1 (t^n)) \, (|\rh{a}| +   |\rh{1}|)
    \\
    &  + \sum_{j=1}^{N-1} \modulo{\rh{j+1} - \rh{j}}
      + \delta_1^n  \, \modulo{\rh{1} - \rh{a}} - \delta_N^n \, \modulo{\rh{N}- \rh{N-1}}
      + \gamma^n_N \, \modulo{\rh{b} - \rh{N}} - \gamma_1^n \, \modulo{\rh{2} - \rh{1}}
    \\
    & + \mathcal{K}_1 (t^n)\, \dt
      \left(\sum_{j=1}^{N-1} \modulo{\rh{j+1} - \rh{j}}
      + \frac12 \, \modulo{\rh{1} - \rh{a}}\right)
      + \mathcal{K}_2 (t^n) \, \dt
      + \mathcal{K}_3 (t^n) \, \dt \, \modulo{\rh{a}}
    \\
    & +|\rho_b^{n+1} - \rh{b}| + (1 - \gamma_N^n) |\rh{b} - \rh{N}|
      + \delta_N^n  |\rh{N}- \rh{N-1}|
      + \dt \, \frac{C}{2} \, (1 + \mathcal{L}\, \mathcal{C}_1 (t^n)) \, (|\rh{N-1}| +
      |\rh{N}|)
    \\
    \leq \
    & |\rho_a^{n+1} -   \rh{a}| +|\rho_b^{n+1} - \rh{b}|
      + \left( 1 + \mathcal{K}_1 (t^n)\, \dt \right)  \sum_{j=0}^{N} \modulo{\rh{j+1} - \rh{j}}
      + \mathcal{K}_2 (t^n) \, \dt
    \\
    & + \frac32 \, C (1 + \mathcal{L}\, \mathcal{C}_1 (t^n)) \, \norma{\rho^n}_{\L\infty (]a,b[)}  \, \dt
      + \left(\mathcal{K}_3 (t^n) +\frac{C}{2} \, (1 + \mathcal{L}\, \mathcal{C}_1 (t^n))  \right)
      \norma{\rho_a}_{\L\infty ([0,t^n])} \, \dt.
  \end{align*}
  Exploiting~\eqref{eq:linf} and setting
  \begin{align}
    \nonumber
    \mathcal{K}_4 (t^n) = \
    & \mathcal{K}_2 (t^n) + \frac32 \, C (1
      + \mathcal{L}\, \mathcal{C}_1 (t^n)) \, e^{\mathcal{C}_2 (t^n)
      \, t^n} \max\left\{\norma{\rho_o}_{\L\infty (]a,b[)}, \,
      \norma{\rho_a}_{\L\infty ([0,t])}, \, \norma{\rho_b}_{\L\infty
      ([0,t])}\right\}
    \\
    \label{eq:K4}
    &+ \left[\mathcal{K}_3 (t^n) +\frac{C}{2} \, (1 + \mathcal{L}\,
      \mathcal{C}_1 (t^n)) \right] \norma{\rho_a}_{\L\infty
      ([0,t^n])},
  \end{align}
  we deduce from the previous estimate by a standard iterative
  procedure
  \begin{align*}
    \sum_{j=0}^N \modulo{\rho_{j+1}^{n}  -\rho_j^{n}} \leq \
    & e^{\mathcal{K}_1 (t^n) \, t^n}  \left(
      \sum_{j=0}^N \modulo{\rho_{j+1}^{0}  -\rho_j^{0}}
      +  \sum_{m=1}^n \modulo{\rho_{a}^{m}  -\rho_a^{m-1}}
      +  \sum_{m=1}^n \modulo{\rho_{b}^{m}  -\rho_b^{m-1}}
      \right)
    \\
    & + \frac{\mathcal{K}_4(t^n)}{\mathcal{K}_1 (t^n)} \,
      \left( e^{\mathcal{K}_1 (t^n) \, t^n} - 1  \right),
  \end{align*}
  concluding the proof.
\end{proof}

\begin{corollary}\label{cor:BVxt} {\bf ($\BV$ estimate in space and time)}
  Let $\rho_o \in \BV (\,]a,b[; \reali^+)$ and
  $\rho_a, \, \rho_b \in \BV (\reali^+; \reali^+)$. Let~\ref{f},
  \ref{omega} and~\eqref{eq:CFL} hold.  Then for all
  $n=1, \ldots, N_T$, the following estimate holds
  \begin{equation}
    \label{eq:BVxt}
    \sum_{m=0}^{n-1} \sum_{j=0}^{N}
    \dt \, \modulo{\rho^{m}_{j+1}-\rho^{m}_{j}}
    +
    \sum_{m=0}^{n-1} \sum_{j=0}^{N+1}
    \dx \, \modulo{\rho_j^{m+1} - \rho_j^m}
    \leq
    \mathcal{C}_{xt} (t^n),
  \end{equation}
  where $\mathcal{C}_{xt} (n \, \dt)$ is given by~\eqref{eq:Cxt}.
\end{corollary}

\begin{proof}
  By Proposition~\ref{prop:BV}, we have
  \begin{equation}
    \label{eq:7}
    \sum_{m=0}^{n-1} \sum_{j=0}^{N}
    \dt \, \modulo{\rho^{m}_{j+1}-\rho^{m}_{j}}
    \leq n \, \dt \, \mathcal{C}_x (n \, \dt).
  \end{equation}
  By the definition of the scheme~\eqref{eq:scheme}, for
  $m \in \left\{0, \ldots, n-1\right\}$ and
  $j\in\left\{1, \ldots, N \right\}$, we have
  \begin{align*}
    \modulo{\rho_j^{m+1} - \rho_j^m} \leq \
    & \frac{\lambda \, \alpha}{2} \left(
      \modulo{\rho_{j+1}^m - \rho_j^m} +
      \modulo{\rho_j^m - \rho_{j-1}^m}
      \right)
    \\
    & + \frac{\lambda}{2}\left|
      f (t^m, x_{j+1/2}, \rho_j^m, R_{j+1/2}^m)
      + f (t^m, x_{j+1/2}, \rho_{j+1}^m, R_{j+1/2}^m)
      \right.
    \\
    &\qquad \left.
      - f (t^m, x_{j-1/2}, \rho_{j-1}^m, R_{j-1/2}^m)
      - f (t^m, x_{j-1/2}, \rho_j^m, R_{j-1/2}^m)
      \right|
    \\
    \leq \
    & \frac{\lambda \, \alpha}{2} \left(
      \modulo{\rho_{j+1}^m - \rho_j^m} +
      \modulo{\rho_j^m - \rho_{j-1}^m}
      \right)
    \\
    & + \frac\lambda2 \left[
      \modulo{\partial_x f (t^m, \tilde{x}_{j}, \rho_{j}^m, R_{j+1/2}^m)} \dx
      \right.
    \\
    &\qquad +
      \modulo{\partial_\rho f (t^m, x_{j-1/2}, \tilde{\rho}_{j-1/2}^m, R_{j+1/2}^m)}\modulo{\rho_j^m - \rho_{j-1}^m}
    \\
    & \qquad
      + \modulo{\partial_R f (t^m, x_{j-1/2}, \rho_{j-1}^m, \tilde{R}_{j}^m)}\modulo{R_{j+1/2}^m - R_{j-1/2}^m}
    \\
    & \qquad +  \modulo{\partial_x f (t^m, \tilde{x}_{j}, \rho_{j+1}^m, R_{j+1/2}^m)} \dx
    \\
    &\qquad+
      \modulo{\partial_\rho f (t^m, x_{j-1/2}, \tilde{\rho}_{j+1/2}^m, R_{j+1/2}^m)}\modulo{\rho_{j+1}^m - \rho_{j}^m}
    \\
    & \qquad \left.
      + \modulo{\partial_R f (t^m, x_{j-1/2}, \rho_{j}^m, \tilde{R}_{j}^m)}\modulo{R_{j+1/2}^m - R_{j-1/2}^m}
      \right]
    \\
    \leq \
    & \frac\lambda2 \, (\alpha + L) \left(
      \modulo{\rho_{j+1}^m - \rho_j^m} +
      \modulo{\rho_j^m - \rho_{j-1}^m}
      \right)
    \\
    & + \frac \lambda2 \, C \dx \left[
      \modulo{\rho_{j+1}^m} + \modulo{\rho_j^m}
      + \mathcal{L}\, \mathcal{C}_1 (t^m) \left(
      \modulo{\rho_j^m} + \modulo{\rho_{j-1}^m}
      \right)
      \right],
  \end{align*}
  where
  \begin{align*}
    \tilde{x}_{j} \in \
    & ]x_{j-1/2}, x_{j+1/2}[,
    &
      \tilde{R}_j^m \in\
    & \mathcal{I} (R_{j-1/2}^m, R_{j+1/2}^m),
    \\
    \tilde{\rho}_{j-1/2}^m \in \
    & \mathcal{I} (\rho_{j-1}^m, \rho_j^m),
    &
      \tilde{\rho}_{j+1/2}^m \in \
    &\mathcal{I} (\rho_{j}^m, \rho_{j+1}^m).
  \end{align*}
  Therefore, by Lemma~\ref{lem:l1} and Proposition~\ref{prop:BV}
  \begin{align*}
    \sum_{j=1}^N  \dx \,  \modulo{\rho_j^{m+1} - \rho_j^m}  \leq \
    &  \dt \, (\alpha + L) \,  \sum_{j=0}^N  \modulo{\rho_{j+1}^m - \rho_j^m}
      +  \dt \, C \, (1 + \mathcal{L}\, \mathcal{C}_1 (m \,\dt))
      \, \dx \, \sum_{j=1}^N \modulo{\rho_j^m}
    \\
    &+ \dt  \, \dx \, \frac{C}{2} \left(\modulo{\rho_b^m} + \mathcal{L} \, \, \mathcal{C}_1 (m \, \dt)
      \modulo{\rho_a^m}\right)
    \\
    \leq \
    &  \dt \, (\alpha + L) \, \mathcal{C}_x (m \, \dt)
      + \dt \, C \, (1 + \mathcal{L}\, \mathcal{C}_1 (m \, \dt)) \, \mathcal{C}_1 (m \, \dt)
    \\
    & + \dt  \, \frac{C}{2} \left(\norma{\rho_b}_{\L\infty ([0,t^m])}
      + \mathcal{L} \, \mathcal{C}_1 (m \, \dt) \,
      \norma{\rho_a}_{\L\infty ([0,t^m])}\right)
    \\
    = \
    & \dt \, \mathcal{C}_t (m \, \dt),
  \end{align*}
  where we set
  \begin{equation}
    \label{eq:Ct}
    \mathcal{C}_t (\tau) =
    (\alpha + L) \, \mathcal{C}_x (\tau)
    +  C \, \mathcal{C}_1 (\tau) \left(1 + \mathcal{L}\, \mathcal{C}_1 (\tau)\right)
    + \frac{C}{2} \left(\norma{\rho_b}_{\L\infty ([0,\tau])}
      + \mathcal{L} \, \mathcal{C}_1 (\tau) \,
      \norma{\rho_a}_{\L\infty ([0,\tau])}\right).
  \end{equation}
  In particular,
  \begin{align}
    \nonumber
    \sum_{j=0}^{N+1}  \dx \,  \modulo{\rho_j^{m+1} - \rho_j^m} =\
    &  \dx \, \modulo{\rho_a^{m+1} - \rho_a^m} + \dx \, \modulo{\rho_b^{m+1} - \rho_b^m}
      + \sum_{j=1}^N  \dx \,  \modulo{\rho_j^{m+1} - \rho_j^m}
    \\
    \label{eq:miserve}
    \leq \
    & \dx \, \modulo{\rho_a^{m+1} - \rho_a^m} + \dx \, \modulo{\rho_b^{m+1} - \rho_b^m} + \dt \, \mathcal{C}_t (m \, \dt),
  \end{align}
  which, summed over $m=0, \ldots, n-1$, yields
  \begin{equation}
    \label{eq:8}
    \sum_{m=0}^{n-1} \sum_{j=0}^{N+1}  \dx \,  \modulo{\rho_j^{m+1} - \rho_j^m} \leq
    \dx  \sum_{m=0}^{n-1} \left( \modulo{\rho_a^{m+1} - \rho_a^m}
      + \modulo{\rho_b^{m+1} - \rho_b^m} \right)
    + n \, \dt \,  \mathcal{C}_t (n \, \dt).
  \end{equation}
  Summing~\eqref{eq:7} and~\eqref{eq:8} we obtain the desired
  estimate~\eqref{eq:BVxt}, with
  \begin{align}
    \nonumber
    \mathcal{C}_{xt} (n \, \dt) = \
    & n \, \dt \, (1 + \alpha + L) \,
      \mathcal{C}_x (n \, \dt)
      + n \, \dt \, C \, \mathcal{C}_1 (n \, \dt) (1+\mathcal{L} \, \mathcal{C}_1 (n \, \dt)) \,
    \\   \label{eq:Cxt}
    & + n \, \dt \, \frac{C}{2} \left(\norma{\rho_b}_{\L\infty
      ([0,t^n])} + \mathcal{L} \, \mathcal{C}_1 (n \, \dt) \, \norma{\rho_a}_{\L\infty
      ([0,t^n])}\right)
    \\ \nonumber
    & + \dx \sum_{m=0}^{n-1} \left( \modulo{\rho_a^{m+1} - \rho_a^m}
      + \modulo{\rho_b^{m+1} - \rho_b^m} \right),
  \end{align}
  concluding the proof. Notice that the last sum in~\eqref{eq:Cxt} is
  bounded by
  \begin{displaymath}\dx \left(\tv(\rho_a;[0,T]) +
      \tv(\rho_b;[0,T])\right).
  \end{displaymath}
\end{proof}

\subsection{Discrete entropy inequality}
\label{sec:Eineq}

We introduce the following notation: for $j=1, \ldots, N$,
$n =0, \ldots, N_T -1$, $k \in \reali$,
\begin{align*}
  H_j^n (u,v,z) = \
    & v -\lambda \, \left( F_{j+1/2}^n (v,z) -  F_{j-1/2}^n (u,v)\right),
  \\
  G^{n,k}_{j+1/2} (u,v) =  \
    & F_{j+1/2}^n (u \wedge k,v \wedge k) -  F_{j+1/2}^n (k,k),
  \\
  L^{n,k}_{j+1/2} (u,v) = \
    & F_{j+1/2}^n (k,k) - F_{j+1/2}^n (u \vee k,v \vee k),
\end{align*}
where $F^n_{j+1/2} (u,v)$ is defined as in~\eqref{eq:numFl}.  Observe
that, due to the definition of the scheme,
$\rho_j^{n+1} = H_j^n (\rh{j-1}, \rh{j}, \rh{j+1})$. Notice moreover
that the following equivalences hold true: $(s-k)^+ = s \wedge k - k$
and $(s-k)^- = k - s \vee k$.

\begin{lemma}
  \label{lem:Eineq}
  Let~\ref{f}, \ref{omega} and~\eqref{eq:CFL} hold. Then the
  approximate solution $\rho_\Delta$ in~\eqref{eq:rhodelta} satisfies
  the following discrete entropy inequalities: for $j=1, \ldots, N$,
  $n =0, \ldots, N_T-1$ and $k \in \reali$,
  \begin{equation}
    \label{eq:Eineq+}
    \begin{aligned}
      (\rho_j^{n+1} - k)^+ - (\rh{j} -k)^+ + \lambda \, \left(
        G^{n,k}_{j+1/2} (\rh{j}, \rh{j+1}) - G^{n,k}_{j-1/2}
        (\rh{j-1}, \rh{j}) \right) &
      \\
      + \lambda \, \sgn^+ (\rho_j^{n+1} - k) \, \left( f (t^n,
        x_{j+1/2}, k, \Rh{j+1/2}) - f (t^n, x_{j-1/2}, k, \Rh{j-1/2})
      \right) & \leq 0,
    \end{aligned}
  \end{equation}
  and
  \begin{equation}
    \label{eq:Eineq-}
    \begin{aligned}
      (\rho_j^{n+1} - k)^- - (\rh{j} -k)^- + \lambda \, \left(
        L^{n,k}_{j+1/2} (\rh{j}, \rh{j+1}) - L^{n,k}_{j-1/2}
        (\rh{j-1}, \rh{j}) \right) &
      \\
      + \lambda \, \sgn^- (\rho_j^{n+1} - k) \, \left( f (t^n,
        x_{j+1/2}, k, \Rh{j+1/2}) - f (t^n, x_{j-1/2}, k, \Rh{j-1/2})
      \right) & \leq 0.
    \end{aligned}
  \end{equation}
\end{lemma}

\begin{proof}
  Consider the map $(u,v,z) \mapsto H_j^n (u,v,z)$. By the CFL
  condition~\eqref{eq:CFL}, it holds
  \begin{align*}
    \frac{\partial H_j^n}{\partial u} (u,v,z) = \
    & \frac{\lambda}{2} \left(
      \partial_\rho f (t^n, x_{j-1/2}, u, \Rh{j-1/2}) + \alpha
      \right)
      \geq 0,
    \\
    \frac{\partial H_j^n}{\partial v} (u,v,z) = \
    & 1 - \lambda \, \alpha - \frac{\lambda}{2}
      \left(
      \partial_\rho f (t^n, x_{j+1/2}, v, \Rh{j+1/2}) -  \partial_\rho f (t^n, x_{j-1/2}, v, \Rh{j-1/2})
      \right)
      \geq 0
    \\
    \frac{\partial H_j^n}{\partial z} (u,v,z) = \
    &\frac{\lambda}{2} \left(
      \alpha - \partial_\rho f (t^n, x_{j+1/2}, z, \Rh{j+1/2})
      \right)
      \geq 0.
  \end{align*}
  Notice that
  \begin{displaymath}
    H_j^n (k,k,k) = k - \lambda \left( f (t^n, x_{j+1/2}, k, \Rh{j+1/2}) - f (t^n, x_{j-1/2}, k, \Rh{j-1/2})\right).
  \end{displaymath}
  The monotonicity properties obtained above imply that
  \begin{align*}
    &  H_j^n (\rh{j-1} \wedge k, \rh{j} \wedge k, \rh{j+1} \wedge k) - H_j^n (k,k,k)
    \\
    \geq \
    & H_j^n (\rh{j-1}, \rh{j}, \rh{j+1}) \wedge H_j^n (k,k,k)  - H_j^n (k,k,k)
    \\
    = \
    & \left(H_j^n (\rh{j-1}, \rh{j}, \rh{j+1}) - H_j^n (k,k,k)\right) ^+
    \\
    =\
    & \left(
      \rho_j^{n+1} - k + \lambda \left( f (t^n, x_{j+1/2}, k, \Rh{j+1/2}) - f (t^n, x_{j-1/2}, k, \Rh{j-1/2})\right)
      \right)^+.
  \end{align*}
  Moreover, we also have
  \begin{align*}
    &  H_j^n (\rh{j-1} \wedge k, \rh{j} \wedge k, \rh{j+1} \wedge k) - H_j^n (k,k,k)
    \\
    = \
    & (\rh{j} \wedge k) - k
    \\
    & - \lambda \left[
      F^n_{j+1/2} (\rh{j} \wedge k, \rh{j+1} \wedge k) - F^n_{j-1/2} (\rh{j-1} \wedge k, \rh{j} \wedge k)
      - F^n_{j+1/2} (k,k) + F^n_{j-1/2}(k,k)
      \right]
    \\
    = \ & (\rh{j} - k)^+ -
          \lambda \left(  G^{n,k}_{j+1/2} (\rh{j}, \rh{j+1}) -  G^{n,k}_{j-1/2} (\rh{j-1}, \rh{j})\right).
  \end{align*}
  Hence,
  \begin{align*}
    & (\rh{j} - k)^+
      - \lambda \left(  G^{n,k}_{j+1/2} (\rh{j}, \rh{j+1}) -  G^{n,k}_{j-1/2} (\rh{j-1}, \rh{j})\right)
    \\
    \geq \
    &  \left(
      \rho_j^{n+1} - k + \lambda \left( f (t^n, x_{j+1/2}, k, \Rh{j+1/2}) - f (t^n, x_{j-1/2}, k, \Rh{j-1/2})\right)
      \right)^+
    \\
    = \
    & \sgn^+\left(  \rho_j^{n+1} - k + \lambda \left( f (t^n, x_{j+1/2}, k, \Rh{j+1/2})
      - f (t^n, x_{j-1/2}, k, \Rh{j-1/2})\right)  \right)
    \\
    & \times \left( \rho_j^{n+1} - k + \lambda \left( f (t^n, x_{j+1/2}, k, \Rh{j+1/2})
      - f (t^n, x_{j-1/2}, k, \Rh{j-1/2})\right)\right)
    \\
    \geq \
    &\left(  \rho_j^{n+1} - k \right) ^+
      + \lambda \,  \sgn^+\left(  \rho_j^{n+1} - k \right)
      \left(f (t^n, x_{j+1/2}, k, \Rh{j+1/2})  - f (t^n, x_{j-1/2}, k, \Rh{j-1/2}) \right),
  \end{align*}
  proving~\eqref{eq:Eineq+}, while~\eqref{eq:Eineq-} is proven in an
  entirely similar way.
\end{proof}

\subsection{Convergence towards an entropy weak solution}
\label{sec:conv}

The uniform $\L\infty$-bound provided by Lemma~\ref{lem:linf} and the
total variation estimate of Corollary~\ref{cor:BVxt} allow to apply
Helly's compactness theorem, ensuring the existence of a subsequence
of $\rho_\Delta$, still denoted by $\rho_\Delta$, converging in $\L1$
to a function $\rho \in \L\infty ([0,T] \times ]a,b[)$, for all
$T>0$. We need to prove that this limit function is indeed an entropy
weak solution to~\eqref{eq:original}, in the sense of
Definition~\ref{def:MV}.

\begin{lemma}
  \label{lem:entropyLimit}
  Let $\rho_o \in \BV (\,]a,b[; \reali^+)$ and
  $\rho_a, \, \rho_b \in \BV (\reali^+; \reali^+)$. Let~\ref{f},
  \ref{omega} and~\eqref{eq:CFL} hold. Then the piecewise constant
  approximate solutions $\rho_\Delta$ in~\eqref{eq:rhodelta} resulting
  from the adapted Lax--Friedrichs scheme~\eqref{eq:scheme} converge,
  as $\dx \to 0$, towards an entropy weak solution of the initial
  boundary value problem~\eqref{eq:original}.
\end{lemma}

\begin{proof}
  We consider the discrete entropy inequality~\eqref{eq:Eineq+}, for
  the positive semi-entropy, and we follow~\cite{DFG}, see
  also~\cite{Vovelle2002}. Add and subtract
  $G^{n,k}_{j+1/2} (\rh{j}, \rh{j})$ in~\eqref{eq:Eineq+} and
  rearrange it as follows
  \begin{align*}
    0 \geq \
    & (\rho_j^{n+1} - k)^+ - (\rh{j} -k)^+
      + \lambda \,  \left(
      G^{n,k}_{j+1/2} (\rh{j}, \rh{j+1}) - G^{n,k}_{j+1/2} (\rh{j}, \rh{j})\right)
    \\
    &+ \lambda \left( G^{n,k}_{j+1/2} (\rh{j}, \rh{j}) -  G^{n,k}_{j-1/2} (\rh{j-1}, \rh{j})
      \right)
    \\
    & + \lambda \, \sgn^+ (\rho_j^{n+1} - k) \, \left(
      f (t^n, x_{j+1/2}, k, \Rh{j+1/2}) - f (t^n, x_{j-1/2}, k, \Rh{j-1/2})
      \right)
  \end{align*}
  Let $\phi \in \Cc1 ([0,T[ \times [a,b]; \reali^+)$ for some $T > 0$,
  multiply the inequality above by $\dx \, \phi (t^n, x_j)$ and sum
  over $j=1, \ldots, N$ and $n \in \naturali$, so to get
  \begin{align}
    \label{eq:9}
    0 \geq \
    & \dx \sum_{n=0}^{+\infty} \sum_{j=1}^N
      \left[ (\rho_j^{n+1} - k )^+ - (\rh{j} - k)^+ \right] \, \phi (t^n, x_j)
    \\
    \label{eq:9a}
    & + \dt  \sum_{n=0}^{+\infty} \sum_{j=1}^N
      \left[
      \left(G^{n,k}_{j+1/2} (\rh{j}, \rh{j+1}) - G^{n,k}_{j+1/2} (\rh{j}, \rh{j})\right)
      \right.
    \\
    \label{eq:9b}
    & \qquad \left.
      -
      \left(G^{n,k}_{j-1/2} (\rh{j-1}, \rh{j}) - G^{n,k}_{j+1/2} (\rh{j}, \rh{j}) \right)
      \right] \, \phi (t^n, x_j)
    \\
    \label{eq:9c}
    & + \dt \sum_{n=0}^{+\infty} \sum_{j=1}^N
      \sgn^+ (\rho_j^{n+1} - k)\!  \left(
      f (t^n, x_{j+1/2}, k, \Rh{j+1/2}) \!-\! f (t^n, x_{j-1/2}, k, \Rh{j-1/2})
      \right) \phi (t^n, x_j).
  \end{align}
  Summing by parts, we obtain
  \begin{align*}
    [\eqref{eq:9}] = \
    & - \dx \sum_{j=1}^N (\rho_j^0 - k)^+ \, \phi (0,x_j)
      - \dx \, \dt  \sum_{n=1}^{+\infty} \sum_{j=1}^N
      (\rh{j} - k)^+ \frac{\phi (t^n, x_j) - \phi (t^{n-1}, x_j)}{\dt}
    \\
    \underset{\dx \to 0^+}{\longrightarrow}
    & - \int_a^b (\rho_o (x) - k)^+ \, \phi (0,x) \d{x}
      - \int_0^{+\infty} \int_a^b (\rho (t,x) - k)^+ \, \partial_t\phi (t,x) \d{x}\d{t},
  \end{align*}
  and
  \begin{align*}
    &[\eqref{eq:9c}]
    \\ = \
    & \dx \, \dt \sum_{n=0}^{+\infty} \sum_{j=1}^N
      \sgn^+ \!(\rho_j^{n+1} - k) \,
      \frac{f (t^n, x_{j+1/2}, k, \Rh{j+1/2}) \!-\! f (t^n, x_{j-1/2}, k, \Rh{j-1/2})}{\dx}
      \, \phi (t^n, x_j)
    \\
    \underset{\dx \to 0^+}{\longrightarrow}
    & \int_0^{+\infty} \int_a^b   \sgn^+ (\rho (t,x) - k)  \,
      \frac{\d{}}{\d{x}} f (t, x, k, R (t,x)) \, \phi (t,x) \d{x} \d{t},
  \end{align*}
  by the Dominated Convergence
  Theorem. Concerning~\eqref{eq:9a}--\eqref{eq:9b}, we get
  \begin{align*}
    [\eqref{eq:9a}-\eqref{eq:9b}] = \
    & \dt  \sum_{n=0}^{+\infty} \sum_{j=1}^N
      \left(G^{n,k}_{j+1/2} (\rh{j}, \rh{j+1}) - G^{n,k}_{j+1/2} (\rh{j}, \rh{j})\right)  \, \phi (t^n, x_j)
    \\
    & - \dt  \sum_{n=0}^{+\infty} \sum_{j=0}^{N-1}
      \left(G^{n,k}_{j+1/2} (\rh{j}, \rh{j+1}) - G^{n,k}_{j+3/2} (\rh{j+1}, \rh{j+1}) \right) \,
      \phi (t^n, x_{j+1})
    \\
    = \
    & \dt  \sum_{n=0}^{+\infty} \sum_{j=1}^{N-1}
      \left[
      \left(G^{n,k}_{j+1/2} (\rh{j}, \rh{j+1}) - G^{n,k}_{j+1/2} (\rh{j}, \rh{j})\right)  \, \phi (t^n, x_j)
      \right.
    \\
    & \qquad \left.
      - \left(G^{n,k}_{j+1/2} (\rh{j}, \rh{j+1}) - G^{n,k}_{j+3/2} (\rh{j+1}, \rh{j+1}) \right) \,
      \phi (t^n, x_{j+1})
      \right]
    \\
    & + \dt  \sum_{n=0}^{+\infty} \left[
      \left(G^{n,k}_{N+1/2} (\rh{N}, \rh{b}) - G^{n,k}_{N+1/2} (\rh{N}, \rh{N})\right)  \, \phi (t^n, x_N)
      \right.
    \\
    & \qquad \left.
      - \left(G^{n,k}_{1/2} (\rh{a}, \rh{1}) - G^{n,k}_{3/2} (\rh{1}, \rh{1}) \right) \, \phi (t^n, x_{1})
      \right]
    \\
    = \
    & T^{int} + T^b = T,
  \end{align*}
  where we set
  \begin{align*}
    T^{int} = \
    & \dt  \sum_{n=0}^{+\infty} \sum_{j=1}^{N-1}
      \left[
      \left(G^{n,k}_{j+1/2} (\rh{j}, \rh{j+1}) - G^{n,k}_{j+1/2} (\rh{j}, \rh{j})\right)  \, \phi (t^n, x_j)
      \right.
    \\
    & \qquad \left.
      - \left(G^{n,k}_{j+1/2} (\rh{j}, \rh{j+1}) - G^{n,k}_{j+3/2} (\rh{j+1}, \rh{j+1}) \right) \,
      \phi (t^n, x_{j+1})
      \right],
    \\
    T^b = \
    &  \dt  \sum_{n=0}^{+\infty} \left[
      \left(G^{n,k}_{N+1/2} (\rh{N}, \rh{b}) - G^{n,k}_{N+1/2} (\rh{N}, \rh{N})\right)  \, \phi (t^n, x_N)
      \right.
    \\
    & \qquad \left.
      - \left(G^{n,k}_{1/2} (\rh{a}, \rh{1}) - G^{n,k}_{3/2} (\rh{1}, \rh{1}) \right) \, \phi (t^n, x_{1})
      \right].
  \end{align*}
  Define
  \begin{equation}
    \label{eq:S}
    \begin{aligned}
      S = \ & - \dx \, \dt \sum_{n=0}^{+\infty} \sum_{j=1}^N
      G^{n,k}_{j+1/2} (\rh{j}, \rh{j}) \, \frac{\phi (t^n,x_{j+1}) -
        \phi (t^n,x_j)}{\dx}
      \\
      &- \alpha \, \dt \sum_{n=0}^{+\infty} \left( (\rh{a} -k)^+ \,
        \phi (t^n,a) + (\rh{b}-k)^+ \, \phi (t^n,b) \right).
    \end{aligned}
  \end{equation}
  Observe that
  \begin{align*}
    G^{n,k}_{j+1/2} (\rh{j}, \rh{j})  = \
    & F^n_{j+1/2} (\rh{j} \wedge k, \rh{j} \wedge k) - F^n_{j+1/2} (k,k)
    \\
    = \
    & f (t^n, x_{j+1/2}, \rh{j}\wedge k, \Rh{j+1/2}) - f (t^n, x_{j+1/2}, k ,\Rh{j+1/2})
    \\
    = \
    & \sgn^+ (\rh{j} - k ) \left(f (t^n, x_{j+1/2}, \rh{j}, \Rh{j+1/2}) - f (t^n, x_{j+1/2}, k ,\Rh{j+1/2}) \right).
  \end{align*}
  It follows then easily that
  \begin{align*}
    S \underset{\dx \to 0^+}{\longrightarrow} \
    & - \int_0^{+\infty}\int_a^b
      \sgn^+ (\rho (t,x) -k)\, \left( f (t, x, \rho, R (t,x)) - f (t, x, k , R (t,x)) \right) \, \partial_x \phi (t,x)
      \d{x}\d{t}
    \\
    & - \alpha \left( \int_0^{+\infty} (\rho_a (t) -k)^+ \, \phi (t,a) \d{t}
      + \int_0^{+\infty} (\rho_b (t) -k)^+ \, \phi (t,b) \d{t}\right).
  \end{align*}
  Let us rewrite $S$~\eqref{eq:S} as follows:
  \begin{align*}
    S = \
    &  - \dt \sum_{n=0}^{+\infty} \sum_{j=1}^N
      G^{n,k}_{j+1/2} (\rh{j}, \rh{j}) \left(\phi (t^n,x_{j+1}) -  \phi (t^n,x_j\right)
    \\
    &- \alpha \, \dt \sum_{n=0}^{+\infty} \left( (\rh{a} -k)^+ \,
      \phi (t^n,a) + (\rh{b}-k)^+ \, \phi (t^n,b) \right)
    \\
    = \
    & \dt \sum_{n=0}^{+\infty} \left(
      \sum_{j=1}^N G^{n,k}_{j+1/2}(\rh{j}, \rh{j}) \, \phi (t^n, x_{j+1})
      - \sum_{j=0}^{N-1} G^{n,k}_{j+3/2}  (\rh{j+1}, \rh{j+1}) \, \phi (t^n, x_{j+1})
      \right)
    \\
    &- \alpha \, \dt \sum_{n=0}^{+\infty} \left( (\rh{a} -k)^+ \,
      \phi (t^n,a) + (\rh{b}-k)^+ \, \phi (t^n,b) \right)
    \\
    = \
    & -\dt \sum_{n=0}^{+\infty} \sum_{j=1}^{N-1}\left(
      G^{n,k}_{j+1/2}(\rh{j}, \rh{j})  - G^{n,k}_{j+3/2}  (\rh{j+1}, \rh{j+1})
      \right) \, \phi (t^n, x_{j+1})
    \\
    & - \dt \sum_{n=0}^{+\infty} \left(
      G^{n,k}_{N+1/2}(\rh{N}, \rh{N}) \,  \phi (t^n, x_{N+1})
      - G^{n,k}_{3/2}  (\rh{1}, \rh{1}) ) \, \phi (t^n, x_{1})
      \right)
    \\
    &- \alpha \, \dt \sum_{n=0}^{+\infty} \left( (\rh{a} -k)^+ \,
      \phi (t^n,a) + (\rh{b}-k)^+ \, \phi (t^n,b) \right)
    \\
    =\
    & S^{int} + S^b,
  \end{align*}
  where we set
  \begin{align*}
    S^{int} = \
    & -\dt \sum_{n=0}^{+\infty} \sum_{j=1}^{N-1}\left(
      G^{n,k}_{j+1/2}(\rh{j}, \rh{j})  - G^{n,k}_{j+3/2}  (\rh{j+1}, \rh{j+1})
      \right) \, \phi (t^n, x_{j+1}),
    \\
    S^b = \
    & - \dt \sum_{n=0}^{+\infty} \left(
      G^{n,k}_{N+1/2}(\rh{N}, \rh{N}) \,  \phi (t^n, x_{N+1})
      - G^{n,k}_{3/2}  (\rh{1}, \rh{1}) ) \, \phi (t^n, x_{1})
      \right)
    \\
    &- \alpha \, \dt \sum_{n=0}^{+\infty} \left( (\rh{a} -k)^+ \,
      \phi (t^n,a) + (\rh{b}-k)^+ \, \phi (t^n,b) \right).
  \end{align*}
  Focus on $S^{int}$: by adding and subtracting
  $G^{n,k}_{j+1/2} (\rh{j}, \rh{j})$ in the brackets, we can rewrite
  it as
  \begin{align*}
    S^{int} = \
    & -\dt \sum_{n=0}^{+\infty} \sum_{j=1}^{N-1}\left(
      G^{n,k}_{j+1/2}(\rh{j}, \rh{j+1})  - G^{n,k}_{j+1/2}  (\rh{j}, \rh{j})
      \right) \, \phi (t^n, x_{j+1})
    \\
    & -\dt \sum_{n=0}^{+\infty} \sum_{j=1}^{N-1}\left(
      G^{n,k}_{j+1/2}(\rh{j}, \rh{j+1})  - G^{n,k}_{j+3/2}  (\rh{j+1}, \rh{j+1})
      \right) \, \phi (t^n, x_{j+1}),
  \end{align*}
  We evaluate now the distance between $T^{int}$ and $S^{int}$:
  \begin{displaymath}
    \modulo{T^{int} - S^{int}} \leq
    \dt \sum_{n=0}^{+\infty} \sum_{j=1}^{N-1}
    \modulo{ G^{n,k}_{j+1/2}(\rh{j}, \rh{j+1})  - G^{n,k}_{j+1/2}  (\rh{j}, \rh{j})}\,
    \modulo{\phi (t^n, x_{j+1}) - \phi (t^n, x_j)}.
  \end{displaymath}
  Observe that
  \begin{align*}
    & \modulo{ G^{n,k}_{j+1/2}(\rh{j}, \rh{j+1})  - G^{n,k}_{j+1/2}  (\rh{j}, \rh{j})}
    \\
    = \
    & \modulo{ F^n_{j+1/2}(\rh{j} \wedge k, \rh{j+1} \wedge k)  - F^n_{j+1/2}  (\rh{j} \wedge k, \rh{j}) \wedge k}
    \\
    =\
    & \frac12 \left|f (t^n, x_{j+1/2},\rh{j}\wedge k, \Rh{j+1/2})
      + f (t^n, x_{j+1/2}, \rh{j+1} \wedge k, \Rh{j+1/2})\right.
    \\
    & \qquad
      \left.
      - 2 \, f (t^n, x_{j+1/2}, \rh{j} \wedge k, \Rh{j+1/2})
      - \alpha \, ( \rh{j+1} \wedge k -  \rh{j} \wedge k)\right|
    \\
    = \
    & \frac12 \left|f (t^n, x_{j+1/2},\rh{j+1}\wedge k, \Rh{j+1/2})
      - f (t^n, x_{j+1/2}, \rh{j} \wedge k, \Rh{j+1/2})\right.
    \\
    & \qquad\qquad
      \left. - \alpha \, ( \rh{j+1} \wedge k -  \rh{j} \wedge k)\right|
    \\
    \leq \
    &  \frac12\, (L + \alpha) \, \modulo{ \rh{j+1} \wedge k -  \rh{j} \wedge k}
    \\
    \leq \
    &  \alpha \, \modulo{ \rh{j+1} -  \rh{j}}.
  \end{align*}
  Therefore,
  \begin{align}
    \nonumber
    \modulo{T^{int} - S^{int}}
    \leq \
    & \alpha \, \dx \, \dt \, \norma{\partial_x \phi}_{\L\infty}
      \sum_{n=0}^{+\infty} \sum_{j=1}^{N-1}
      \modulo{ \rh{j+1} -  \rh{j}}
    \\
    \label{eq:int}
    \leq \
    & \alpha \, \dx \, T \, \norma{\partial_x \phi}_{\L\infty}\max_{0 \leq n \leq T/\dt} \tv (\rho_\Delta (t^n, \cdot))
      = \mathcal{O} (\dx),
  \end{align}
  thanks to the uniform BV estimate~\eqref{eq:spaceBV}. Pass now to
  the terms $T^b$ and $S^b$:
  \begin{align}
    \nonumber
    S^b - T^b = \
    &  - \dt \sum_{n=0}^{+\infty} \left(
      G^{n,k}_{N+1/2}(\rh{N}, \rh{N}) \,  \phi (t^n, x_{N+1})
      - G^{n,k}_{3/2}  (\rh{1}, \rh{1}) ) \, \phi (t^n, x_{1})
      \right)
    \\ \nonumber
    &- \alpha \, \dt \sum_{n=0}^{+\infty} \left( (\rh{a} -k)^+ \, \phi (t^n,a) + (\rh{b}-k)^+ \, \phi (t^n,b) \right)
    \\ \nonumber
    &- \dt  \sum_{n=0}^{+\infty}
      \left(G^{n,k}_{N+1/2} (\rh{N}, \rh{b}) - G^{n,k}_{N+1/2} (\rh{N}, \rh{N})\right)  \, \phi (t^n, x_N)
    \\ \nonumber
    & + \dt \sum_{n=0}^{+\infty}
      \left(G^{n,k}_{1/2} (\rh{a}, \rh{1}) - G^{n,k}_{3/2} (\rh{1}, \rh{1}) \right) \, \phi (t^n, x_{1})
    \\ \label{eq:10}
    = \
    & \dt  \sum_{n=0}^{+\infty}
      \left( G^{n,k}_{1/2} (\rh{a}, \rh{1}) \, \phi (t^n, x_{1}) - \alpha \,  (\rh{a} -k)^+ \, \phi (t^n,a)
      \right)
    \\ \label{eq:10b}
    & - \dt \sum_{n=0}^{+\infty}
      \left( \alpha \,  (\rh{b}-k)^+ \, \phi (t^n,b)  + G^{n,k}_{N+1/2} (\rh{N}, \rh{b}) \,  \phi (t^n, x_N)\right)
    \\
    \label{eq:10c}
    & - \dt  \sum_{n=0}^{+\infty}G^{n,k}_{N+1/2}(\rh{N}, \rh{N}) \left(\phi (t^n, x_{N+1})- \phi (t^n, x_N)\right).
  \end{align}
  Observe that
  \begin{align*}
    \frac{\partial F^n_{j+1/2}}{\partial u} (u,v)
    = \
    & \frac12 \, \left(
      \partial_\rho f (t^n, x_{j+1/2}, u, \Rh{j+1/2}) + \alpha
      \right)
      \geq \frac12 \, (-L + \alpha)
      \geq 0,
    \\
    \frac{\partial F^n_{j+1/2}}{\partial v} (u,v)
    = \
    & \frac12 \, \left(
      \partial_\rho f (t^n, x_{j+1/2}, v, \Rh{j+1/2}) - \alpha
      \right)
      \leq \frac12 \, (L-\alpha)
      \leq 0,
  \end{align*}
  meaning that the numerical flux is increasing with respect to the
  first variable and decreasing with respect to the second one. Thus,
  \begin{align*}
    G^{n,k}_{j+1/2} (u,v) = \
    & F^n_{j+1/2} (u \wedge k, v \wedge k) - F^n_{j+1/2} (k,k)
    \\
    \geq \
    & F^n_{j+1/2} (k, v \wedge k) - F^n_{j+1/2} (k,k)
    \\
    = \
    & \frac12 \, \left(f (t^n, x_{j+1/2}, v \wedge k, \Rh{j+1/2}) - f (t^n, x_{j+1/2},k, \Rh{j+1/2})
      -\alpha  (v \wedge k - k)
      \right)
    \\
    \geq \
    & - \frac {L+\alpha}{2} \, \modulo{v \wedge k - k}
    \\
    \geq \
    & - \alpha \,  (v-k)^+
  \end{align*}
  and
  \begin{align*}
    G^{n,k}_{j+1/2} (u,v) = \
    & F^n_{j+1/2} (u \wedge k, v \wedge k) - F^n_{j+1/2} (k,k)
    \\
    \leq \
    & F^n_{j+1/2} (u \wedge k, k) - F^n_{j+1/2} (k,k)
    \\
    = \
    & \frac12 \, \left(f (t^n, x_{j+1/2}, u \wedge k, \Rh{j+1/2}) - f (t^n, x_{j+1/2},k, \Rh{j+1/2})
      -\alpha  (k - u \wedge k )
      \right)
    \\
    \leq \
    &  \frac {L+\alpha}{2} \, \modulo{u \wedge k - k}
    \\
    \leq \
    & \alpha \,  (u-k)^+.
  \end{align*}
  Hence,
  \begin{align*}
    [\eqref{eq:10}] = \
    &  \dt  \sum_{n=0}^{+\infty}
      G^{n,k}_{1/2} (\rh{a}, \rh{1})  \left(\phi (t^n, x_{1}) - \phi (t^n,a)\right)
    \\
    &  + \dt  \sum_{n=0}^{+\infty}
      \left( G^{n,k}_{1/2} (\rh{a}, \rh{1}) - \alpha \,  (\rh{a} -k)^+ \right) \phi (t^n,a)
    \\
    \leq \
    & \alpha \, T \, \dx \, \norma{\partial_x \phi}_{\L\infty} \sup_{0\leq n \leq T/\dt}(\rh{a} - k)^+
    \\
    & + \alpha \, \dt \sum_{n=0}^{+\infty} \left( (\rh{a} - k)^+ - (\rh{a} - k)^+\right) \phi (t^n,a)
    \\
    \leq \
    & \alpha \, T \, \dx \, \norma{\partial_x \phi}_{\L\infty} \norma{\rho_a}_{\L\infty} = \mathcal{O} (\dx),
    \\
    [\eqref{eq:10b}] = \
    &  - \dt  \sum_{n=0}^{+\infty}
      \left(  \alpha \,  (\rh{b} -k)^+  + G^{n,k}_{N+1/2} (\rh{N}, \rh{b})\right) \phi (t^n,b)
    \\
    &
      - \dt  \sum_{n=0}^{+\infty}
      G^{n,k}_{N+1/2} (\rh{N}, \rh{b})  \left(\phi (t^n, x_{N}) - \phi (t^n,b)\right)
    \\
    \leq \
    & - \alpha \, \dt \sum_{n=0}^{+\infty} \left( (\rh{b} - k)^+ - (\rh{b} - k)^+\right) \phi (t^n,b)
    \\
    & +  \alpha \, T \, \dx \, \norma{\partial_x \phi}_{\L\infty} \sup_{0\leq n \leq T/\dt}(\rh{b} - k)^+
    \\
    \leq \
    & \alpha \, T \, \dx \, \norma{\partial_x \phi}_{\L\infty} \norma{\rho_b}_{\L\infty} = \mathcal{O} (\dx),
    \\
    [\eqref{eq:10c}] = \
    & \dt \modulo{
      \sum_{n=0}^{+\infty}G^{n,k}_{N+1/2}(\rh{N}, \rh{N}) \left(\phi (t^n, x_{N+1})- \phi (t^n, x_N)\right)
      }
    \\
    \leq \
    & \dt \, \dx \, \norma{\partial_x \phi}_{\L\infty} \sum_{n=0}^{+\infty}
      \modulo{G^{n,k}_{N+1/2}(\rh{N}, \rh{N}) }
    \\
    = \
    &\dt \, \dx \, \norma{\partial_x \phi}_{\L\infty} \sum_{n=0}^{+\infty}
      \modulo{F^n_{N+1/2} (\rh{N} \wedge k, \rh{N} \wedge k) - F^N_{N+1/2} (k,k)}
    \\
    = \
    & \dt \, \dx \, \norma{\partial_x \phi}_{\L\infty} \sum_{n=0}^{+\infty}
      \modulo{f (t^n, x_{N+1/2}, \rh{N} \wedge k, \Rh{N+1/2}) - f (t^n, x_{N+1/2}, k, \Rh{N+1/2})}
    \\
    \leq
    & L \, \dt \, \dx \, \norma{\partial_x \phi}_{\L\infty} \sum_{n=0}^{+\infty} \modulo{\rh{N} \wedge k - k}
    \\
    = \
    &  L \, \dt \, \dx \, \norma{\partial_x \phi}_{\L\infty} \sum_{n=0}^{+\infty} (\rh{N} - k)^+
    \\
    \leq \
    & L \, T \, \dx \, \norma{\partial_x \phi}_{\L\infty} \sup_{0\leq n \leq T/ \dt} \norma{\rho^n }_{\L\infty}
      = \mathcal{O} (\dx),
  \end{align*}
  thanks to the uniform $\L\infty$ estimate~\eqref{eq:linf}.  Hence,
  $S^b - T^b \leq \mathcal{O} (\dx)$, so that we finally get
  \begin{align*}
    0 \geq \
    & [\eqref{eq:9}\ldots\eqref{eq:9c}]
    \\
    = \
    & [\eqref{eq:9}] + [\eqref{eq:9c}] + T \pm S
    \\
    \geq \
    & [\eqref{eq:9}] + [\eqref{eq:9c}] + S - \mathcal{O} (\dx),
  \end{align*}
  concluding the proof.
\end{proof}

\section{Lipschitz continuous dependence on initial and boundary data}
\label{sec:lipDep}

\begin{proposition}
  \label{prop:lipDep}
  Fix $T>0$. Let~\ref{f}, \ref{omega} and~\eqref{eq:CFL} hold.  Assume
  moreover
  $\partial_{x\rho}^2f, \, \partial_{\rho R}^2 f \in \L\infty ([0,T]
  \times [a,b] \times \reali^2;\reali)$. Let
  $\rho_o, \sigma_o \in \BV (\,]a,b[; \reali^+)$ and
  $\rho_a, \, \rho_b, \, \sigma_a, \, \sigma_b \in \BV (\,]0,T[;
  \reali^+)$. Call $\rho$ and $\sigma$ the corresponding solutions
  to~\eqref{eq:original}. Then the following estimate holds
  \begin{align*}
    & \norma{\rho (T) - \sigma (T)}_{\L1 (]a,b[)}
    \\
    \leq \
    &  \left(\norma{\rho_o - \sigma_o}_{\L1 (]a,b[)}
      + L \left(
      \norma{\rho_a - \sigma_a}_{\L1 ([0,T])} +
      \norma{\rho_b - \sigma_b}_{\L1 ([0,T])}
      \right)\right)
      \left(
      1 + B(T) \,T \, e^{B(T) \, T}
      \right),
  \end{align*}
  and $B(T)$ is defined in~\eqref{eq:Bt}.
\end{proposition}

\begin{proof}
  Introduce the following notation:
  \begin{equation}
    \label{eq:32}
    \begin{aligned}
      R (t,x) = \ & \frac{1}{W (x)} \, \int_a^b \rho (t,y) \, \omega
      (y-x) \d{y}, &&& g (t,x,u) = \ & f (t,x,u,R (t,x)),
      \\
      S (t,x) = \ & \frac{1}{W (x)} \, \int_a^b \sigma (t,y) \, \omega
      (y-x) \d{y}, &&& h (t,x,u) = \ &f (t,x,u,S (t,x)).
    \end{aligned}
  \end{equation}
  Observe that, due to~\ref{omega} and Theorem~\ref{thm:main}, for any
  $t \in [0,T]$ and $x \in [a,b]$,
  \begin{align*}
    R(t,x)
    \leq \
    & \frac{\norma{\omega}_{\L\infty}}{K_\omega} \, \mathcal{R}_1(t),
    &&&
        S(t,x)
        \leq \
    & \frac{\norma{\omega}_{\L\infty}}{K_\omega} \, \mathcal{S}_1(t),
  \end{align*}
  where $\mathcal{S}_1(t)$ is defined analogously to
  $\mathcal{R}_1(t)$ in~\eqref{eq:R1} for $\sigma$.  For later use,
  set
  \begin{displaymath}
    J(t)=
    \frac{\norma{\omega}_{\L\infty}}{K_\omega} \,
    \max\left\{\mathcal{R}_1(t), \,\mathcal{S}_1(t)\right\},
  \end{displaymath}
  so that $R(t,x), \, S(t,x) \in [-J(t), J(t)]$. Compute for later use
  \begin{align}
    \label{eq:Rx}
    & \modulo{\partial_x R (t,x)} \leq \
      \left(\frac{\norma{\omega'}_{\L1}\norma{\omega}_{\L\infty}}{K_\omega^2}
      + \frac{\norma{\omega'}_{\L\infty}}{K_\omega}
      \right) \norma{\rho (t)}_{\L1 (]a,b[)}
      \leq \mathcal{L}\, \mathcal{R}_1(t),
    \\
    \nonumber
    & \modulo{\partial_{xx}^2 R (t,x)}
    \\
    \nonumber
    \leq \
    &\left(
      \frac{2 \, \norma{\omega'}_{\L1}^2  \norma{\omega}_{\L\infty}}{K_\omega^3}
      + \frac{\norma{\omega''}_{\L1} \norma{\omega}_{\L\infty} }{K_\omega^2}
      + \frac{2 \,  \norma{\omega'}_{\L1} \norma{\omega'}_{\L\infty}}{K_\omega^2}
      + \frac{\norma{\omega''}_{\L\infty} }{K_\omega}
      \right) \norma{\rho (t)}_{\L1 (]a,b[)}
    \\
    \label{eq:Rxx}
    \leq \
    & \mathcal{W} \,  \mathcal{R}_1 (t),
  \end{align}
  with $\mathcal{L}$ defined exactly as in~\eqref{eq:L} and
  $\mathcal{W}$ defined as in~\eqref{eq:Wconst}. Observe that
  $\mathcal{L}$ and $\mathcal{W}$ are finite thanks
  to~\ref{omega}. Compute also
  \begin{align}
    \label{eq:RS}
    \modulo{R (t,x) - S (t,x)} \leq \ &
                                        \frac{\norma{\omega}_{\L\infty}}{K_\omega}
                                        \int_a^b \modulo{\rho (t,y) - \sigma (t,y)}\d{y},
    \\
    \label{eq:RSx}
    \modulo{\partial_x R (t,x) - \partial_x S (t,x)}
    \leq \ & \mathcal{L}
             \int_a^b \modulo{\rho (t,y) - \sigma (t,y)}\d{y}.
  \end{align}

  We can think of $\rho$ and $\sigma$ as solutions to the following
  IBVPs
  \begin{displaymath}
    \left\{
      \begin{array}{l}
        \del_t \rho + \del_x g(t,x,\rho) =0,
        \\
        \rho(0,x) = \rho_o(x),
        \\
        \rho(t,a) = \rho_a (t),
        \\
        \rho(t,b) = \rho_b (t),
      \end{array}
    \right.
    \qquad\quad
    \left\{
      \begin{array}{l@{\qquad}r@{\,}c@{\,}l}
        \del_t \sigma + \del_x h(t,x,\sigma) =0,
        & (t,x)
        &\in
        & ]0,T[ \times ]a,b[,
        \\
        \sigma(0,x) = \sigma_o(x),
        & x
        &\in
        &]a,b[,
        \\
        \sigma(t,a) = \sigma_a (t),
        & t
        &\in
        &]0,T[,
        \\
        \sigma(t,b) = \sigma_b (t),
        & t
        &\in
        &]0,T[.
      \end{array}
    \right.
  \end{displaymath}
  Moreover, consider also the following IBVP:
  \begin{equation}
    \label{eq:misto}
    \left\{
      \begin{array}{l@{\qquad}r@{\,}c@{\,}l}
        \del_t \pi + \del_x g(t,x,\pi) =0,
        & (t,x)
        &\in
        & ]0,T[ \times ]a,b[,
        \\
        \pi(0,x) = \sigma_o(x),
        & x
        &\in
        &]a,b[,
        \\
        \pi(t,a) = \sigma_a (t),
        & t
        &\in
        & ]0,T[,
        \\
        \pi(t,b) = \sigma_b (t),
        & t
        &\in
        & ]0,T[.
      \end{array}
    \right.
  \end{equation}
  Thanks to~\ref{f} and to the additional assumptions on $f$, the flux
  functions $g$ and $h$ defined in~\eqref{eq:32} satisfy the
  hypotheses of Theorem~\ref{thm:bohEU}, Proposition~\ref{prop:bohLip}
  and Theorem~\ref{thm:bohStab}. Indeed, focusing on $g$, compute:
  \begin{displaymath}
    \partial_{xu}^2 g (t,x,u) = \partial_{x\rho}^2 f(t,x,u,R (t,x)
    + \partial_{u R}^2 f (t,x,u,R (t,x)) \, \partial_x R (t,x).
  \end{displaymath}
  Thus, thanks also to~\eqref{eq:Rx}, $\partial_{xu}^2 g$ is
  finite. Therefore, we can use the results of~\cite{1d}, recalled in
  Appendix~\ref{sec:app}: by Theorem~\ref{thm:bohEU},
  problem~\eqref{eq:misto} admits a unique solution in
  $(\L\infty \cap \BV) (\,]0,T[ \times ]a,b[; \reali)$, which
  satisfies, for all $t\in[0,T[$
  \begin{align*}
    \norma{\pi (t)}_{\L\infty (]a,b[)}\leq \
    & \mathcal{P}(t) ,
    \\
    \tv (\pi (t)) \leq \
    & \left(
      \tv(\sigma_o) + \tv(\sigma_a;[0,t]) + \tv(\sigma_b;[0,t]) + \mathcal{K}(t) \, t
      \right) e^{\mathcal{C}_4(t) \, t},
  \end{align*}
  with
  \begin{align*}
    \mathcal{P}(t) := \
    & \left(\max\left\{\norma{\sigma_o}_{\L\infty (]a.b[)}, \,
      \norma{\sigma_a}_{\L\infty ([0,t])}, \, \norma{\sigma_b}_{\L\infty ([0,t])}\right\}
      + \mathcal{C}_3(t) \, t \right)e^{\mathcal{C}_4(t) \, t},
    \\
    \mathcal{C}_3(t) := \
    & \norma{\partial_x g (\cdot, \cdot, 0)}_{\L\infty ([0,t] \times [a,b])},
    \\
    \mathcal{C}_4(t) := \
    & \norma{\partial_{xu}^2 g}_{\L\infty ([0,t] \times [a,b] \times \reali)},
    \\
    \mathcal{K} (t) := \
    & 2 \, \mathcal{C}_3(t)
      + 2 \, (b-a) \, \norma{\partial_{xx}^2 g}_{\L\infty([0,t] \times [a,b] \times [-\mathcal{P}(t), \, \mathcal{P}(t)])}
    \\
    & + \frac12 \left( 3 \, \mathcal{P}(t) + \norma{\sigma_a}_{\L\infty([0,t])}\right)
      \norma{\partial_{xu}^2 g}_{\L\infty([0,t] \times [a,b] \times [-\mathcal{P}(t), \, \mathcal{P}(t)])}.
  \end{align*}
  Due to the definition of $g$ in~\eqref{eq:32} and~\ref{f}, we obtain
  \begin{align*}
    \modulo{\partial_x g (t, x, 0)}
    =\
    & \modulo{\partial_x f (t,z,0, R (t,x))
      + \partial_R f (t,x,0,R (t,x)) \, \partial_x R (t,x)}
      = 0,
    \\
    \norma{\partial_{xu}^2 g}_{\L\infty ([0,t] \times [a,b] \times \reali)} \leq \
    & \norma{\partial_{x \rho}^2 f}_{\L\infty ([0,t] \times [a,b] \times \reali \times [-J(t), J(t)])}
    \\
    & +
      \norma{\partial_{\rho R}^2 f}_{\L\infty ([0,t] \times [a,b] \times \reali \times [-J(t), J(t)])}
      \mathcal{L} \, \mathcal{R}_1(t),
  \end{align*}
  so that $\mathcal{C}_3(t)=0$,
  $\mathcal{C}_4(t) \leq \mathcal{C}_5(t)$, where we set
  \begin{equation}
    \label{eq:c5}
    \mathcal{C}_5 (t) =
    \norma{\partial_{x \rho}^2 f}_{\L\infty ([0,t] \times [a,b] \times \reali \times [-J(t), J(t)])}
    +
    \norma{\partial_{\rho R}^2 f}_{\L\infty ([0,t] \times [a,b] \times \reali \times [-J(t), J(t)])}
    \mathcal{L} \, \mathcal{R}_1(t)
  \end{equation}
  Hence
  \begin{equation}
    \label{eq:pinf}
    \norma{\pi (t)}_{\L\infty (]a,b[)}\leq
    \mathcal{P}_{\infty} (t) =
    e^{\mathcal{C}_5
      (t) \, t}
    \max\left\{\norma{\sigma_o}_{\L\infty (]a.b[)}, \,
      \norma{\sigma_a}_{\L\infty ([0,t])}, \, \norma{\sigma_b}_{\L\infty ([0,t])}\right\}.
  \end{equation}
  Moreover,
  \begin{align*}
    & \norma{\partial_{xx}^2 g}_{\L\infty([0,t] \times [a,b] \times [-\mathcal{P}(t), \, \mathcal{P}(t)])}
    \\
    \leq \
    & \norma{\partial_{xx}^2 f}_{\L\infty([0,t] \times [a,b] \times [-\mathcal{P}(t), \, \mathcal{P}(t)] \times [-J(t), J(t)])}
    \\
    & +
      2 \, \norma{\partial_{x R}^2 f}_{\L\infty([0,t] \times [a,b] \times [-\mathcal{P}(t), \, \mathcal{P}(t)] \times [-J(t), J(t)])} \, \norma{\partial_{x} R}_{\L\infty([0,t] \times[a,b])}
    \\
    & +  \norma{\partial_{R R}^2 f}_{\L\infty([0,t] \times [a,b] \times [-\mathcal{P}(t), \, \mathcal{P}(t)] \times [-J(t), J(t)])}\,  \norma{\partial_{x} R}^2_{\L\infty([0,t] \times[a,b])}
    \\
    & + \norma{\partial_{R} f}_{\L\infty([0,t] \times [a,b] \times [-\mathcal{P}(t), \, \mathcal{P}(t)] \times [-J(t), J(t)])}\, \norma{\partial_{xx}^2 R}_{\L\infty([0,t] \times[a,b])}
    \\
    \leq \
    & C \, \mathcal{P}_{\infty}(t) \left(
      1 +
      \mathcal{R}_1 (t) \left( 2 \, \mathcal{L}
      + \mathcal{L}^2 \, \mathcal{R}_1 (t)
      +  \mathcal{W}
      \right)
      \right),
  \end{align*}
  so that $\mathcal{K}(t) \leq \hat{\mathcal{K}} (t)$ where we set
  \begin{equation}
    \label{eq:hatK}
    \hat{\mathcal{K}} (t) =
    2 \, (b-a) \, C \, \mathcal{P}_{\infty}(t) \!\left[
      1 +
      \mathcal{R}_1 (t) \left( 2 \, \mathcal{L}
        + \mathcal{L}^2 \, \mathcal{R}_1 (t)
        +  \mathcal{W}
      \right)
    \right]\!
    + \frac12\left(3 \, \mathcal{P}_{\infty} (t) + \norma{\sigma_a}_{\L\infty ([0,t])}\right)
    \mathcal{C}_3 (t),
  \end{equation}
  and we then obtain
  \begin{equation}
    \label{eq:ptv}
    \tv(\pi (t)) \leq
    \left(
      \tv(\sigma_o) + \tv(\sigma_a;[0,t]) + \tv(\sigma_b;[0,t])
      + \hat{\mathcal{K}}(t) \, t
    \right) e^{\mathcal{C}_5(t) \, t}.
  \end{equation}

  For $t>0$, compute
  \begin{equation}
    \label{eq:20}
    \norma{\rho (t) - \sigma (t)}_{\L1 (]a,b[)} \leq
    \norma{\rho (t) - \pi (t)}_{\L1 (]a,b[)}
    +
    \norma{\pi (t) - \sigma (t)}_{\L1 (]a,b[)}.
  \end{equation}
  The first term on the right hand side of~\eqref{eq:20} evaluates the
  distance between solutions to IBVPs of the type considered in the
  Appendix~\ref{sec:app} with the same flux function, but different
  initial and boundary data. Therefore, we can apply
  Proposition~\ref{prop:bohLip}, to get
  \begin{align*}
    \norma{\rho (t) - \pi (t)}_{\L1 (]a,b[)} \leq \
    & \norma{\rho_o- \sigma_o}_{\L1 (]a,b[)}
    \\
    &+ \norma{\partial_u g}_{\L\infty ([0,t]\times [a,b] \times \reali)}
      \left( \norma{\rho_a - \sigma_a}_{\L1 ([0,t])} +
      \norma{\rho_b - \sigma_b}_{\L1 ([0,t])} \right).
  \end{align*}
  Due to the definition of $g$ in~\eqref{eq:32}, we obtain
  \begin{align*}
    \norma{\partial_u g}_{\L\infty([0,t]\times [a,b] \times \reali)} =
    \norma{\partial_\rho f}_{\L\infty([0,t]\times [a,b] \times
    \reali\times\reali)} < L,
  \end{align*}
  hence
  \begin{equation}
    \label{eq:20aOK}
    \norma{\rho (t) - \pi (t)}_{\L1 (]a,b[)} \leq
    \norma{\rho_o - \sigma_o}_{\L1 (]a,b[)}
    + L \left(
      \norma{\rho_a - \sigma_a}_{\L1 ([0,t])} +
      \norma{\rho_b - \sigma_b}_{\L1 ([0,t])}
    \right).
  \end{equation}

  On the other hand, the second term on the right hand side
  of~\eqref{eq:20} evaluates the distance between solutions to IBVPs
  of the type considered in Appendix~\ref{sec:app} with different flux
  functions, but same initial and boundary data. We apply
  Theorem~\ref{thm:bohStab} to obtain
  \begin{align}
    \nonumber
    & \norma{\pi (t) - \sigma (t)}_{\L1 (]a,b[)}
    \\ \label{eq:30}
    \leq \
    & \int_0^t\int_a^b \norma{\partial_x (g-h) (s,x, \cdot)}_{\L\infty (U (s))}\d{x}\d{s}
    \\ \label{eq:31}
    & + \int_0^t \norma{\partial_u (g-h) (s, \cdot, \cdot)}_{\L\infty (]a,b[\times U (s))} \,
      \min\left\{
      \tv  \left(\sigma (s)\right),  \,\tv  \left(\pi (s)\right)
      \right\}\d{s}
    \\ \label{eq:33}
    & + 2 \int_0^t \norma{(g-h)(s,a, \cdot)}_{\L\infty(U (s))} \d{s}
      +2 \int_0^t \norma{(g-h)(s,b, \cdot)}_{\L\infty(U(s))} \d{s},
  \end{align}
  where $U (s) = [-\mathcal{U} (s), \mathcal{U} (s)]$, with
  $\mathcal{U} (s) = \max\left\{\norma{\pi (s)}_{\L\infty (]a,b[)}, \,
    \norma{\sigma (s)}_{\L\infty (]a,b[)}\right\}$.  Let us now
  estimate all the terms appearing in~\eqref{eq:30}--\eqref{eq:33}.
  First of all, by Theorem~\ref{thm:main},
  \begin{equation}
    \label{eq:Sinf}
    \norma{\sigma (t)}_{\L\infty (]a,b[)} \leq \mathcal{S}_{\infty} (t) =
    e^{t \, C \, \left(1 + \mathcal{L} \, \mathcal{S}_1 (t)\right)}
    \max\left\{\norma{\sigma_o}_{\L\infty (]a.b[)}, \,
      \norma{\sigma_a}_{\L\infty ([0,t])}, \, \norma{\sigma_b}_{\L\infty ([0,t])}\right\},
  \end{equation}
  so that, comparing~\eqref{eq:Sinf} with~\eqref{eq:pinf} we obtain
  \begin{equation}
    \label{eq:Us}
    \mathcal{U} (t) \leq  \max\left\{\norma{\sigma_o}_{\L\infty (]a.b[)}, \,
      \norma{\sigma_a}_{\L\infty ([0,t])}, \, \norma{\sigma_b}_{\L\infty ([0,t])}\right\}
    \exp\left( t\, C \, \left(1 + \mathcal{L} \, \mathcal{S}_1 (t)\right)
      + t \,\mathcal{C}_5(t)\right).
  \end{equation}
  Then, by Theorem~\ref{thm:main},
  \begin{equation}
    \label{eq:stv}
    \tv (\sigma(t)) \leq
    e^{t \, \mathcal{T}_1 (t)} \left(\tv (\sigma_o) + \tv (\sigma_a; [0,t])
      + \tv (\sigma_b;[0,t])\right)
    + \frac{\mathcal{T}_2 (t)}{\mathcal{T}_1 (t)} (e^{t \, \mathcal{T}_1 (t)} -1)
  \end{equation}
  with
  \begin{align*}
    \mathcal{T}_1(t) =\
    & \norma{\partial_{\rho x}^2 f}_{\L\infty ([0,t] \times [a,b]\times \reali^2 )}
      +  \mathcal{L}\, \mathcal{S}_1 (t) \,
      \norma{\partial_{\rho R}^2 f}_{\L\infty ([0,t] \times [a,b]\times \reali^2 )},
    \\
    \mathcal{T}_2(t) = \
    & \mathcal{K}_2 (t)
      + \frac32 \, C (1 + \mathcal{L}\, \mathcal{S}_1 (t)) \,
      \mathcal{S}_\infty (t)
      + \left[\mathcal{K}_3 (t) +\frac{C}{2} \, (1 + \mathcal{L}\,
      \mathcal{S}_1 (t)) \right] \norma{\sigma_a}_{\L\infty ([0,t])},
    \\
    \mathcal{K}_2 (t) = \
    & C \, \mathcal{S}_1 (t) \, \left(
      1 + 2 \, \mathcal{L}\, \mathcal{S}_1 (t)  +2 \, \mathcal{K}_3 (t)
      \right),
    \\
    \mathcal{K}_3 (t) = \
    & C \, \mathcal{S}_1 (t) \left(\mathcal{L}^2 \, \mathcal{S}_1 (t)
      + \frac{1}{2} \,  \mathcal{W} \, \right).
  \end{align*}
  Hence, comparing~\eqref{eq:ptv} and~\eqref{eq:stv}, we get
  \begin{align}
    \nonumber
    \min\left\{
    \tv  \left(\sigma (s)\right),  \,\tv  \left(\pi (s)\right)
    \right\}
    \leq \
    &
      e^{t \, \mathcal{T}_3 (t)} \left(\tv (\sigma_o) + \tv (\sigma_a; [0,t])
      + \tv (\sigma_b;[0,t])\right)
    \\
    \label{eq:T4}
    & + \min\left\{\hat{\mathcal{K}(t)} \, e^{\mathcal{C}_5(t) \, t} , \,
      \frac{\mathcal{T}_2 (t)}{\mathcal{T}_1 (t)} (e^{t \, \mathcal{T}_1 (t)} -1)
      \right\}
    \\
    \nonumber
    =: \
    & \mathcal{T}_4(t),
  \end{align}
  with
  \begin{equation}
    \label{eq:T3}
    \mathcal{T}_3(t) =
    \norma{\partial_{\rho x}^2 f}_{\L\infty ([0,t] \times [a,b]\times \reali^2 )}
    +  \mathcal{L}\, \min\left\{\mathcal{R}_1(t), \,\mathcal{S}_1 (t) \right\}
    \norma{\partial_{\rho R}^2 f}_{\L\infty ([0,t] \times [a,b]\times \reali^2 )}.
  \end{equation}

  Focus on~\eqref{eq:30}: by~\ref{f}, \eqref{eq:RS}
  and~\eqref{eq:RSx},
  \begin{align*}
    & \modulo{\partial_x\left( g(t,x,u) - h(t,x,u)\right)}
    \\
    = \
    & \modulo{\frac{\d{}}{\d{x}}
      \left(f\left(t,x,u,R (t,x)\right) - f\left(t,x,u,S(t,x)\right)\right) }
    \\
    \leq \
    &\modulo{ \partial_x f (t,x,u,R(t,x)) - \partial_x f (t,x,u,S(t,x)) }
    \\
    & + \modulo{ \partial_R f (t,x,u,R(t,x)) \, \partial_x R (t,x)
      - \partial_R f (t,x,u,S(t,x))\, \partial_x S (t,x)}
    \\
    \leq \
    & \modulo{\partial_{xR}^2 f (t,x,u, R_1(t,x))}  \modulo{R (t,x) - S (t,x)}
    \\
    & + \modulo{\partial_{RR}^2 f (t,x,u,R_2(t,x))} \modulo{\partial_x R (t,x)}
      \modulo{R (t,x) - S (t,x)}
    \\
    & + \modulo{ \partial_R f (t,x,u,S(t,x)) }
      \modulo{\partial_x R (t,x) - \partial_x S (t,x)}
    \\
    \leq \
    & C \, \modulo{u} \, \left[ \frac{\norma{\omega}_{\L\infty}}{K_\omega}
      \left( 1 + \mathcal{L} \, \mathcal{R}_1 (t)\right)
      + \mathcal{L}
      \right]
      \, \int_a^b \modulo{\rho (t,y) - \sigma (t,y)}\d{y},
  \end{align*}
  with $R_i(t,x) \in \mathcal{I}\left(R (t,x), S (t,x)\right)$,
  $i=1,2$.  Hence,
  \begin{equation}
    \label{eq:30ok}
    \norma{\partial_x(g-h) (s, x,\cdot)}_{\L\infty (U(s))}
    \leq
    C \, \mathcal{U} (s) \left[ \frac{\norma{\omega}_{\L\infty}}{K_\omega}
      \left( 1 + \mathcal{L} \, \mathcal{R}_1 (s)\right)
      + \mathcal{L}
    \right] \int_a^b \modulo{\rho (s,y) - \sigma (s,y)}\d{y}.
  \end{equation}
  Pass to~\eqref{eq:31}: by~\ref{f}, \eqref{eq:RS} and~\eqref{eq:RSx},
  \begin{align*}
    \modulo{\partial_u\left( g(t,x,u) - h(t,x,u)\right)}
    = \
    &
      \modulo{ \partial_\rho \left(f\left(t,x,u,R (t,x)\right) -
      f\left(t,x,u,S(t,x)\right)\right)}
    \\
    = \
    & \modulo{\partial_{\rho R}^2 f\left(t,x,u,R_3(t,x)\right)}
      \modulo{R (t,x) - S (t,x)},
  \end{align*}
  and therefore
  \begin{equation}
    \label{eq:31ok}
    \norma{\partial_u (g-h) (s, \cdot,\cdot)}_{\L\infty (]a,b[ \times U(s))}
    \leq
    \norma{\partial_{\rho R}^2 f}_{\L\infty ([0,s] \times [a,b] \times U (s) \times \reali )}
    \frac{\norma{\omega}_{\L\infty}}{K_\omega}\!
    \int_a^b\! \modulo{\rho (s,y) - \sigma (s,y)}\d{y}.
  \end{equation}
  Finally, consider the first integral in~\eqref{eq:33}: by~\ref{f}
  and~\eqref{eq:RS}
  \begin{align*}
    \modulo{g (t,a,u) - h(t,a,u)} = \
    & \modulo{f\left(t,a,u,R(t,a)\right) - f\left(t,a,u,S(t,a)\right)}
    \\
    \leq\
    &\modulo{\partial_R f\left(t,a,u, \tilde{R} (t,a)\right)}
      \modulo{R(t,a) -S(t,a)}
    \\
    \leq \
    & C \modulo{u} \, \frac{\norma{\omega}_{\L\infty}}{K_\omega}\!
      \int_a^b \modulo{ \rho(t,y) - \sigma (t,y)} \d{y}
  \end{align*}
  with $\tilde{R}(t,a) \in \mathcal{I}\left(R (t,a), S (t,a)\right)$,
  so that
  \begin{equation}
    \label{eq:33ok}
    \norma{(g-h)(s,a, \cdot)}_{\L\infty (U(s))} \leq
    C \, \mathcal{U} (s) \,  \frac{\norma{\omega}_{\L\infty}}{K_\omega}
    \int_a^b \modulo{ \rho(s,y) - \sigma (s,y)} \d{y},
  \end{equation}
  and similarly for the second integral in~\eqref{eq:33}.

  Collecting together~\eqref{eq:30ok}, \eqref{eq:31ok}
  and~\eqref{eq:33ok}, exploiting~\eqref{eq:Us} and~\eqref{eq:T4}, we
  obtain
  \begin{align}
    \nonumber
    & \norma{\pi (t) - \sigma (t)}_{\L1 (]a,b[)}
    \\
    \nonumber
    \leq \
    &  \Bigl\{(b-a) \,  C \, \mathcal{U} (t)
      \left[ \frac{\norma{\omega}_{\L\infty}}{K_\omega}
      \left( 1 + \mathcal{L} \, \mathcal{R}_1 (t)\right)
      + \mathcal{L}
      \right]
      +
      \norma{\partial_{\rho R}^2 f}_{\L\infty ([0,t] \times [a,b] \times U (t) \times \reali )}
      \frac{\norma{\omega}_{\L\infty}}{K_\omega} \,
      \mathcal{T}_4(t)
    \\
    \label{eq:20bOK}
    & \quad+
      4 \, C \, \mathcal{U} (t) \,  \frac{\norma{\omega}_{\L\infty}}{K_\omega}\Bigr\}
      \int_0^t \int_a^b \modulo{\rho (s,y) - \sigma (s,y)}\d{y} \d{s}.
  \end{align}

  Insert now~\eqref{eq:20aOK} and~\eqref{eq:20bOK} into~\eqref{eq:20}:
  \begin{displaymath}
    \norma{\rho (t) - \sigma (t)}_{\L1 (]a,b[)} \leq
    A(t) + B(t)  \int_0^t  \norma{\rho (s) - \sigma (s)}_{\L1 (]a,b[)} \d{s},
  \end{displaymath}
  where
  \begin{align}
    \nonumber
    A(t) = \
    & \norma{\rho_o - \sigma_o}_{\L1 (]a,b[)}
      + L \left(
      \norma{\rho_a - \sigma_a}_{\L1 ([0,t])} +
      \norma{\rho_b - \sigma_b}_{\L1 ([0,t])}
      \right),
    \\
    \label{eq:Bt}
    B(t) = \
    &  (b-a) \,  C \, \mathcal{U} (t)
      \left[ \frac{\norma{\omega}_{\L\infty}}{K_\omega}
      \left( 1 + \mathcal{L} \, \mathcal{R}_1 (t)\right)
      + \mathcal{L}
      \right]
      + 4 \, C \, \mathcal{U} (t) \,  \frac{\norma{\omega}_{\L\infty}}{K_\omega}
    \\
    \nonumber
    & +
      \norma{\partial_{\rho R}^2 f}_{\L\infty ([0,t] \times [a,b] \times U (t) \times \reali )}
      \frac{\norma{\omega}_{\L\infty}}{K_\omega}\,
      \mathcal{T}_4(t)
  \end{align}
  with $\mathcal{U}(t)$ as in~\eqref{eq:Us}, $\mathcal{L}$ as
  in~\eqref{eq:L}, $\mathcal{R}_1$ as in~\eqref{eq:R1},
  $\mathcal{T}_4(t)$ as in~\eqref{eq:T4}.  An application of
  Gronwall's Lemma yields the desired estimate:
  \begin{align*}
    \norma{\rho (t) - \sigma (t)}_{\L1 (]a,b[)}
    \leq \ &
             A(t) +
             \int_0^t A(s) B(s) e^{\int_s^t B(\tau) \d \tau} \d s \\
    \leq \
           & A(t) + B(t)
             \int_0^t A(s) e^{B(t) (t-s)} \d s
    \\
    \leq \
           &
             A(t) \, \left(
             1 + B(t) \, t \, e^{B(t) \, t}
             \right).
  \end{align*}
\end{proof}

\section{Proof of Theorem~\ref{thm:main}}
\label{sec:proof}
\begin{proofof}{Theorem~\ref{thm:main}}
  The existence of solutions to problem~\eqref{eq:original} follows
  from the results of Section~\ref{sec:existence}, in particular
  \S~\ref{sec:conv}. The uniqueness is ensured by the Lipschitz
  continuous dependence of solutions to~\eqref{eq:original} on initial
  and boundary data, see Section~\ref{sec:lipDep}.

  The estimates on the solution to~\eqref{eq:original} are obtained
  from the corresponding discrete estimates passing to the limit. In
  particular, the $\L1$ bound follows from~\eqref{eq:l1}, the
  $\L\infty$ bound from~\eqref{eq:linf}, the total variation bound
  from~\eqref{eq:spaceBV} and the Lipschitz continuity in time
  from~\eqref{eq:miserve}, since $\dx = \frac{\dt}{\lambda}$ and
  taking $\lambda= \frac{1}{3 \, L}$.
\end{proofof}

\begin{appendices}
  \section{The local 1D IBVP}
  \label{sec:app}

  We recall below some results concerning the classical (local) one
  dimensional initial boundary value problem for a scalar conservation
  laws. Detailed proofs can be found in~\cite{1d}, which deals with
  the more general case of a balance law.

  Fix $T>0$, set $I = \, ]0,T[$ and consider the IBVP:
  \begin{equation}\label{eq:boh}
    \left\{
      \begin{array}{l@{\qquad}r@{\,}c@{\,}l}
        \del_t u + \d{}_x f(t,x,u) =0,
        & (t,x)
        &\in
        & I\times ]a,b[,
        \\
        u(0,x) = u_o(x),
        & x
        &\in
        &]a,b[,
        \\
        u(t,a) = u_a (t),
        & t
        &\in
        & I,
        \\
        u(t,b) = u_b (t),
        & t
        &\in
        & I.
      \end{array}
    \right.
  \end{equation}
  Above, the notation for $\d{}_x f\left(t,x,u(t,x)\right)$ follows
  closely that introduced in~\eqref{eq:Dx}, that is:
  \begin{displaymath}
    \d{}_x f \left( t,x,u(t,x)\right) = \partial_x f \left( t,x,u(t,x)\right)
    + \partial_u f  \left( t,x,u(t,x)\right) \, \partial_x u(t,x).
  \end{displaymath}

  Recall the definition of solution to~\eqref{eq:boh}. In particular,
  we focus on the adaptation to the present one dimensional setting of
  the definition of solution provided by Bardos, le Roux and
  N\'ed\'elec~\cite[p.~1028]{BardosLerouxNedelec}.
  \begin{definition}
    \label{def:BLN}
    A function $u\in (\L\infty \cap \BV) (I \times ]a,b[; \reali)$ is
    an entropy weak solution to problem~\eqref{eq:boh} if for all test
    function $\phi \in \Cc1 (\,]-\infty,T[ \times \reali; \reali^+)$
    and $k\in \reali$
    \begin{align*}
      & \int_0^T \int_{a}^{b}  \left\{ \modulo{u(t,x)  -k} \partial_t \phi (t,x)
        + \sgn\left(u (t,x) - k\right) \left[
        f\left(t, x, u(t,x)\right) - f\left(t, x, k\right)
        \right] \, \partial_x \phi (t,x)
        \right.
      \\
      & \qquad\qquad \left.
        - \sgn\left(u(t,x) -k\right) \partial_x f \left(t, x, k\right) \,  \phi (t,x) \right\}
        \d{x} \d{t}
      \\
      & + \int_{a}^{b}  \modulo{u_o(x)-k}  \phi(0,x)  \d{x}
      \\
      & + \int_0^T
        \sgn \left(  u_a(t)  - k \right)
        \left[
        f \left(t, a, u (t, a^+)\right) - f\left(t, a, k\right)
        \right]\,  \phi(t,a)  \d{t}
      \\
      & - \int_0^T
        \sgn \left(  u_b(t)  - k \right)
        \left[
        f\left(t, b, u (t,b^-) -  f \left(t, b, k\right) \right)
        \right] \,\phi(t,b) \d{t}
        \geq  0.
    \end{align*}
  \end{definition}

  The well-posedness of problem~\eqref{eq:boh}, some \emph{a priori}
  estimates on its solution and the stability of its solution with
  respect to variations in the flux function are proved
  in~\cite{1d}. We report the results below, adapted to the present
  setting without source term. 

  \begin{theorem}{\cite[Theorem~2.4]{1d}}\label{thm:bohEU}
    Let $f \in \C2([0,T] \times [a,b] \times \reali;\reali)$, with
    $\partial_u f, \, \partial_{xu}^2 f \in \L\infty ([0,T] \times
    [a,b] \times \reali;\reali)$. Let
    $u_o \in \BV(\,]a,b[; \reali^+)$,
    $u_a, \, u_b \in \BV(I; \reali^+)$.

    Then the IBVP~\eqref{eq:boh} has a unique solution
    $u \in (\L\infty \cap \BV)(I \times ]a,b[; \reali)$, satisfying
    \begin{align}
      \label{eq:bohinf}
      & \norma{u(t)}_{\L\infty(]a,b[)} \leq \mathcal{U}(t),
      \\
      \label{eq:bohtv}
      & \tv(u(t)) \leq e^{\mathcal{C}_4(t) \, t} \left( \tv(u_o) +
        \tv(u_a;[0,t]) + \tv(u_b;[0,t]) + \mathcal{K}(t) \, t
        \right),
    \end{align}
    for any $t \in [0,T[$, with
    \begin{align*}
      \mathcal{U}(t) = \
      & \left(
        \max\left\{\norma{u_o}_{\L\infty(]a,b[)}, \,
        \norma{u_a}_{\L\infty([0,t])},
        \norma{u_b}_{\L\infty([0,t])}\right\} + \mathcal{C}_3(t)
        \, t \right) e^{\mathcal{C}_4(t) \, t},
      \\
      \mathcal{C}_3(t) = \
      & \norma{\partial_x f (\cdot, \cdot, 0)}_{\L\infty([0,t] \times [a,b])},
      \\
      \mathcal{C}_4(t) = \
      & \norma{\partial_{xu}^2 f}_{\L\infty([0,t] \times [a,b] \times \reali)},
      \\
      \mathcal{K}(t) = \
      & 2 \, \mathcal{C}_3(t) + 2 \, (b-a) \, \norma{\partial_{xx}^2 f}_{\L\infty([0,t] \times [a,b] \times [-\mathcal{U}(t), \, \mathcal{U}(t)])}
      \\
      & + \frac12 \left( 3 \, \mathcal{U}(t) +
        \norma{u_a}_{\L\infty([0,t])}\right) \norma{\partial_{xu}^2
        f}_{\L\infty([0,t] \times [a,b] \times [-\mathcal{U}(t), \,
        \mathcal{U}(t)])}.
    \end{align*}
  \end{theorem}

\begin{proposition}{\cite[Proposition~3.7]{1d}}\label{prop:bohLip}
  Let $f \in \C2([0,T] \times [a,b] \times \reali;\reali)$, with
  $\partial_u f, \, \partial_{xu}^2 f \in \L\infty ([0,T] \times [a,b]
  \times \reali;\reali)$. Let $u_o,\, v_o \in\BV (\,]a,b[; \reali^+)$,
  $u_a, \, u_b, \, v_a, \, v_b \in \BV(I; \reali^+)$. Call $u$ and $v$
  the corresponding solutions to the IBVP~\eqref{eq:boh}. Then, for
  all $t>0$, the following estimate holds
  \begin{align*}
    \norma{u(t)-v(t)}_{\L1(]a,b[)} \leq \
    & \norma{u_o - v_o}_{\L1(]a,b[)}
    \\
    & + \norma{\partial_u f}_{\L\infty([0,t] \times [a,b] \times \reali)}
      \left( \norma{u_a - v_a}_{\L1([0,t])} +  \norma{u_b - v_b}_{\L1([0,t])} \right).
  \end{align*}
\end{proposition}

\begin{theorem}{\cite[Theorem~2.6]{1d}}\label{thm:bohStab}
  Let $f_1, f_2 \in \C2([0,T] \times [a,b] \times \reali;\reali)$,
  with $\partial_u
  f_i$,$ \partial_{xu}^2 f_i \in \L\infty ([0,T] \times [a,b] \times
  \reali;\reali)$ for $i=1,2$. Let $u_o \in \BV (\,]a,b[; \reali^+)$,
  $u_a, \, u_b \in \BV (I; \reali^+)$.  Call $u_1$ and $u_2$ the
  corresponding solutions to the IBVP~\eqref{eq:boh}. Then, for
  $t\in [0,T[$, the following estimate holds
  \begin{align*}
    &  \norma{u_1 (t) - u_2 (t)}_{\L1 (]a,b[)}
    \\ \leq \
    &\int_0^t \int_a^b \norma{\partial_x (f_2-f_1) (s,x, \cdot)}_{\L\infty (U(s))} \d{x}\d{s}
    \\
    & + \int_0^t \norma{\partial_u (f_2-f_1) (s, \cdot, \cdot)}_{\L\infty (]a,b[\times U(s))} \,
      \min\left\{\tv \left(u_1 (s)\right), \, \tv \left(u_2 (s)\right)\right\} \d{s}
    \\
    & + 2 \int_0^t  \norma{(f_2-f_1)(s,a, \cdot)}_{\L\infty(U(s))} \d{s}
      +2  \int_0^t \norma{(f_2-f_1)(s,b, \cdot)}_{\L\infty(U(s))} \d{s},
  \end{align*}
  where, with the notation introduced in Theorem~\ref{thm:bohEU},
  $\norma{u_i(s)}_{\L\infty(]a,b[)} \leq \mathcal{U}_i(s)$, for
  $i=1,2$, and
  \begin{displaymath}
    U(s) = [-\mathcal{U}(s), \, \mathcal{U}(s)],
    \quad
    \mbox{ with }
    \quad
    \mathcal{U}(s) = \max_{i=1,2} \, \mathcal{U}_i(s).
  \end{displaymath}
\end{theorem}

\end{appendices}

\small{ \bibliography{ibvp2}

  \bibliographystyle{abbrv} }

\end{document}